\documentclass[12pt, a4paper, reqno]{article}
\usepackage{amsmath, amsfonts, amsthm}
\usepackage[utf8]{inputenc}
\usepackage[T1]{fontenc}
\usepackage{bbm}

\usepackage[dvipsnames]{xcolor}
\usepackage[polish, english]{babel}

\usepackage{dsfont}
\usepackage{hyperref}
\usepackage{enumitem}
\usepackage{ulem}
\usepackage{cancel}
\usepackage{todonotes}

\newtheorem{theorem}{Theorem}[section]
\newtheorem{definition}[theorem]{Definition}
\newtheorem{corollary}[theorem]{Corollary}
\newtheorem{lemma}[theorem]{Lemma}
\newtheorem{proposition}[theorem]{Proposition}

\theoremstyle{remark}
\newtheorem{remark}[theorem]{Remark}
\newtheorem{example}{Example}[section]

\def \RR {\mathbb{R}}
\def \Rd {\mathbb{R}^d}

\def \tp {\tilde{p}}

\renewcommand{\d}{{\rm d}}

\title{Invariant and Super-median Functions for Perturbed Heat Kernels}

\author{%
  Tomasz Jakubowski\thanks{Wydział Matematyki, Politechnika Wrocławska,
  Wyb. Wyspiańskiego 27, 50-370 Wrocław, Poland.
  E-mail: \texttt{tomasz.jakubowski@pwr.edu.pl}}
  \and
  Karol Szczypkowski\thanks{Wydział Matematyki, Politechnika Wrocławska,
  Wyb. Wyspiańskiego 27, 50-370 Wrocław, Poland.
  E-mail: \texttt{karol.szczypkowski@pwr.edu.pl}}
}

\date{}

\begin{document}

\maketitle

\begin{abstract}
We study super-median and invariant functions for positive and negative Schr\"odinger perturbations of non-negative kernels. For positive perturbations, combining Dynkin's formula for the original kernel with Duhamel's formula for the perturbation, we derive perturbed Dynkin-type identities and give integrability conditions ensuring invariance. 
When the perturbed kernel is given by the perturbation series, we obtain a sufficient condition for a function to be super-median without additional integrability assumptions.
We also decompose every finite non-negative solution of Duhamel's equation into the perturbation series and a non-negative solution of the corresponding homogeneous equation. 
For negative perturbations, we obtain analogous identities under an explicit finiteness condition. We also construct functions satisfying Dynkin's formula using zero-resolvents and compensated kernels. Examples involving Gaussian, fractional, and Bessel kernels illustrate the invariance criteria and the role of non-uniqueness in Duhamel's equation.
\end{abstract}

\section{Introduction}

A heat kernel (or transition density) provides a description of an evolution (of a particle or a system) at an intermediate level between its pathwise realizations and its averaged representation by the corresponding semigroup.
A common way of obtaining a new evolution model from a given one is by adding a potential term. The analysis of the perturbed evolution becomes particularly challenging when the resulting heat kernel is no longer comparable with the original one.

Before turning to a specific example, let us describe the general idea behind this problem. Let $(P_t)_{t\geq 0}$ be a semigroup with generator $L$, and suppose that a positive function $h$ is super-median, meaning that $P_t h\leq h$. Formally, this corresponds to the super-harmonicity condition $Lh\leq0$, as suggested by Dynkin's formula
\[
P_t h-h=\int_0^t P_sLh\,ds.
\]
The function $h$ then determines the non-negative potential
\[
V=-\frac{Lh}{h},
\]
and hence the perturbed operator $\widetilde L=L+V$, for which $h$ is harmonic, since $\widetilde Lh=0$. It is therefore natural to expect $h$ to be invariant under the semigroup $(\widetilde P_t)_{t\geq0}$ associated with $\widetilde L$, namely that $\widetilde P_t h=h$. Such invariant functions play an important role in estimates of the corresponding heat kernel $\tp$. However, as we shall see, this expected invariance may fail. Our aim is to understand when Dynkin's formula leads to invariance under the corresponding perturbation.

To examine these issues in a specific setting, let $p$ be the heat kernel of the isotropic $\alpha$-stable process in $\RR^d$,
 where $d\in\mathbb N$ and $\alpha\in(0,2)$, that is, the fundamental solution to the equation $\partial_t u=-(-\Delta)^{\alpha/2}u$. The fundamental solution $\tp$ of $\partial_t u=-(-\Delta)^{\alpha/2}u+\kappa|x|^{-\alpha}u$ may be treated as the perturbation of $p$ by a potential $V(x)=\kappa|x|^{-\alpha}$, where $\kappa$ is a constant. Such a perturbation is known to be critical; in particular, the value of $\kappa$ strongly affects the existence and  behaviour of non-negative $\tp$.  
For this example, assume that $d>\alpha$.
Put
\[
\kappa_\beta=\frac{2^\alpha\Gamma\big(\frac{\beta+\alpha}{2}\big)\Gamma\big(\frac{d-\beta}{2}\big)}
{\Gamma\big(\frac{\beta}{2}\big)\Gamma\big(\frac{d-\beta-\alpha}{2}\big)},
\qquad 0<\beta<d-\alpha,
\]
and set $\kappa_0=0$ by continuity.
In~\cite{MR3460023}, the authors show the finiteness of $\tp$ by proving that
\begin{equation}\label{eq:i-5}
\int_{\RR^d}\tp(t,x,y)|y|^{-\delta}\,\d{y}\leq|x|^{-\delta},
\qquad t>0,\ x\in\RR^d,
\end{equation}
where  the value of $\delta\in(0,(d-\alpha)/2]$ is uniquely determined by the relation $\kappa_\delta = \kappa$.
In fact, the authors propose a general method of deriving inequalities such as~\eqref{eq:i-5}, inspired by the equality
\begin{equation}\label{eq:i-6}
|x|^{-\delta}=\int_0^\infty\int_{\RR^d}f(t)p(t,x,y)\,\mu(\d{y})\d{t},
\qquad x\in\RR^d,
\end{equation}
where $f(t)=ct^{(d-\alpha-\delta)/\alpha}$, $c$ is a positive constant depending on the parameters, and $\mu(\d{y})=\delta_0(\d{y})$ is the Dirac measure at the origin. Later, in~\cite{MR3933622}, the result was improved to
\begin{equation}\label{eq:i-1}
\int_{\RR^d}\tp(t,x,y)|y|^{-\delta}\,\d{y}=|x|^{-\delta},
\qquad t>0,\ x\in\RR^d,
\end{equation}
which is an instance of a more general family of equalities (see~\cite[Theorem~3.1]{MR3933622}): for $\beta\in[0,d-\alpha-\delta)$,
\begin{equation}\label{eq:i-2}
\int_{\RR^d}\tp(t,x,y)|y|^{-\beta}\,\d{y}
+(\kappa_\beta-\kappa)\int_0^t\int_{\RR^d}\tp(s,x,y)|y|^{-\beta-\alpha}\,\d{y}\,ds
=|x|^{-\beta}, \qquad t>0,\ x\ne0 .
\end{equation}
We note that~\eqref{eq:i-1} can be fully recovered from~\eqref{eq:i-2} for the entire range $\delta\in (0,(d-\alpha)/2]$.
However, this can be done directly by setting $\beta=\delta$ only for $\delta\in (0,(d-\alpha)/2)$; in the endpoint case $\delta=(d-\alpha)/2$, additional arguments are required. Furthermore, the term involving $(\kappa_\beta-\kappa)$ is in fact a difference of two integrals, one containing $\kappa_\beta$ and the other $\kappa$. The finiteness of these integrals is both non-trivial and crucial. The equalities~\eqref{eq:i-2} are important for deriving sharp two-sided estimates for $\tp$ (see \cite[Theorem~1.1]{MR3933622}): 
\begin{align*}
\tilde{p}(t,x,y) \approx \left(1\land\frac{|x|}{t^{1/\alpha}}\right)^{-\delta}\left(1\land\frac{|y|}{t^{1/\alpha}}\right)^{-\delta}  p(t,x,y), \qquad t>0,\; x,y \in \RR^d,
\end{align*}
where the weight $|x|^{-\delta}$ determines the correction factors.
The use of weights in heat kernel estimates for Hardy-type perturbations is also central to the method of desingularizing weights, developed by Milman and Semenov~\cite{MR1994762, MR2064932, MR2114706} and applied in the fractional setting by Kinzebulatov and Semenov~\cite{kinzebulatov2019}.
In the local case,
the role of positive harmonic functions in such estimates is investigated
in~\cite{MR3684109}.
The proof of~\eqref{eq:i-2} 
starts by establishing that, for $\beta\in[0,d-\alpha)$,
\begin{equation}\label{eq:i-3}
\int_{\RR^d}p(t,x,y)|y|^{-\beta}\,\d{y}
+\kappa_\beta\int_0^t\int_{\RR^d}p(s,x,y)|y|^{-\beta-\alpha}\,\d{y}\,\d{s}
=|x|^{-\beta}, \qquad t>0,\ x\in\RR^d,
\end{equation}
see~\cite[Lemma~3.3]{MR3933622}. This follows from \eqref{eq:i-6} with $\beta$ in place of $\delta$, since \eqref{eq:i-6} is  valid throughout the range $(0,d-\alpha)$. 
In particular, $|x|^{-\beta}$ is super-median for $p$,
and the potential
associated with the weight $|x|^{-\beta}$ is obtained from~\eqref{eq:i-3} as the ratio $\kappa_\beta|x|^{-\beta-\alpha}/|x|^{-\beta}=\kappa_\beta|x|^{-\alpha}$. Finally, $p$ and $\tp$ are related by Duhamel's formula,
\begin{equation}\label{eq:i-4}
\tp(t,x,y)=p(t,x,y)+\int_0^t\int_{\RR^d}\tp(s,x,z)\kappa|z|^{-\alpha}p(t-s,z,y)\,\d{z}\,\d{s}, \qquad t>0,\ x,y\in\RR^d.
\end{equation}

The success of this approach, leading to a refined analysis of $\tp$, is in part due to the fact that it does not rely on any functional-analytic framework, which could limit the flexibility of other methods. Consequently, it was later applied to operators beyond the fractional Laplacian, for which counterparts of~\eqref{eq:i-3}, \eqref{eq:i-2}, and~\eqref{eq:i-1} were 
obtained by adapting the original arguments.
 In many of these cases, including those considered in~\cite{MR4660833,MR4937881} and \cite{MR4746984}, the kernels $p$ and $\tp$ possess scaling properties, which facilitate the analysis. A notable exception is the relativistic fractional Laplacian, studied under critical perturbations in~\cite{MR4705624,MR4645729}. Since this operator is known to lack scaling, its treatment involves additional difficulties and also relies on the results for the fractional Laplacian~\cite{MR3933622}.

The above considerations are closely related to the theory of Hardy inequalities and identities. Fitzsimmons~\cite{MR1786080} and Rao and \v{S}iki\'c~\cite{MR2215185} link such inequalities for Dirichlet forms to positive super-harmonic functions. The construction in~\cite{MR3460023} formulates this connection in terms of symmetric transition densities and their weighted time integrals.
Under additional assumptions, this construction yields Hardy identities, called ground-state representations, which decompose the Dirichlet form into a potential term and a non-negative remainder involving the quotient of a test function by a chosen positive weight. 
Ground-state representations for fractional Schr{\"o}dinger operators and more general non-local forms were studied in~\cite{MR2425175, MR2469027}, \cite{MR3460023, MR4746984}, \cite{MR4660833}.
The corresponding Hardy inequality follows by discarding this remainder. 
Moreover, the weight is super-median for the perturbation series associated with the corresponding potential. This motivates our study of conditions under which such weights are also invariant for the perturbed kernel.

A central question is to understand, in a general setting, the origin of~\eqref{eq:i-6} and its connection with Dynkin's formula~\eqref{eq:i-3}, see Lemma~\ref{ex:1}. A natural approach to deriving the perturbed identities~\eqref{eq:i-1} and~\eqref{eq:i-2} is then to combine~\eqref{eq:i-3} with Duhamel's formula~\eqref{eq:i-4}, see Theorem~\ref{prop:tp-init} and Corollary~\ref{cor:eq_ineq}.
At this point, however, a genuine obstruction arises: even when a function satisfies Dynkin’s formula for the unperturbed kernel, invariance under the perturbed kernel cannot be expected in general.
Recall that~\eqref{eq:i-6}, with $\beta$ in place of $\delta$, and~\eqref{eq:i-3} are valid throughout the range $\beta\in(0,d-\alpha)$.
The exclusion of the value $d-\alpha-\delta$ from the range of $\beta$ in~\eqref{eq:i-2} is not merely a technical restriction. Indeed, we still have
\[
\int_{\RR^d}\tp(t,x,y)|y|^{-(d-\alpha-\delta)}\,\d{y}\leq|x|^{-(d-\alpha-\delta)},
\qquad t>0,\ x\in\RR^d,
\]
yet equality does not hold if $\delta\in[0,(d-\alpha)/2)$, see~\cite[Corollary~6.3]{MR4595537}, even though
\[
\kappa_{d-\alpha-\delta}-\kappa_\delta=0.
\]
Thus, for $\delta\in(0,(d-\alpha)/2)$, the potential $\kappa|x|^{-\alpha}$ is associated with the two weights $|x|^{-\delta}$ and  $|x|^{\delta-(d-\alpha)}$. 
Although both of them are obtained from~\eqref{eq:i-6} by suitable choices of $f$, the latter is super-median for $\tp$ but not invariant.

In the present paper, we study Dynkin-type identities under positive and negative Schr\"odinger perturbations of non-negative kernels. Our approach separates the algebraic consequences of Duhamel's formula from the additional integrability conditions that suffice for invariance, and makes the role of finite cancellation explicit. For positive perturbations, an additional issue is the distinction between the perturbation series and other non-negative solutions of Duhamel's equation~\eqref{eq:tp-Duh}.

\subsection*{Results}
First we discuss positive perturbations. Theorem~\ref{prop:tp-init} provides a perturbed Dynkin-type identity, while Corollary~\ref{cor:eq_ineq} gives a practical condition ensuring invariance.
In earlier studies of Hardy perturbations, proving invariance required a sequence of lemmas tailored to the kernel under consideration; see, for example, \cite{bogdan2024}.
Our results considerably simplify this procedure: once $\Phi$ satisfies Dynkin's formula for $p$  and $\phi$, and $\phi(\d{y})=\Phi(y)V(\d{y})$, invariance for the perturbation series follows  from the single integrability condition~\eqref{eq:eq_ineq} holding for all $t>0$ and $x\in X$.
Without~\eqref{eq:eq_ineq}, the equality $\phi(\d{y})=\Phi(y)V(\d{y})$ alone does not determine the behaviour of $\Phi$ under $\tp$; in particular, $\Phi$ may be invariant, super-median, or sub-median, see Remark~\ref{rem:3}. The same criterion applies to any non-negative kernel $\tp$ satisfying Duhamel's formula for $p$ and $V$. The argument requires neither symmetry nor the Chapman--Kolmogorov property.

For the perturbation series, Theorem~\ref{thm:super_median} shows that $\Phi$ is super-median whenever $\Phi(y)V(\d{y})\leq\phi(\d{y})$, without the additional integrability assumption~\eqref{eq:eq_ineq}. Here the choice of the perturbation series is essential: it is the minimal non-negative solution of Duhamel's equation, whereas $\Phi$ need not be super-median for other non-negative solutions.
The structure of these other solutions is described in Lemma~\ref{lem:omega4}: every finite non-negative solution of Duhamel's equation is the sum of the perturbation series and a non-negative solution of the corresponding homogeneous equation.
Two examples illustrate different aspects of this non-uniqueness. For the Bessel kernels on the half-line, different solutions corresponding to the same inverse-square potential may have different invariant power functions, see Corollaries~\ref{cor:1} and~\ref{cor:2}, as well as Remark~\ref{rem:3}. For the fractional Laplacian on the half-line, we show that
$
\Phi(x)-\int_0^\infty\tp(t,x,y)\Phi(y)\d{y}
$
is a non-trivial solution of the homogeneous equation for a suitable choice of $\Phi$; in the critical case, the corresponding weight includes a logarithmic factor, see Lemma~\ref{lem:psi-construction} and Example~\ref{ex:9}.

For negative perturbations, we start from a sub-median version of Dynkin's formula \eqref{eq:init-n} and assume that a non-negative kernel satisfying Duhamel’s formula \eqref{eq:Duh-n} for $p$ and $-U$ is available.
Theorem~\ref{prop:tp-init-n} gives the corresponding perturbed Dynkin-type identity under the integrability condition~\eqref{ineq:finite-n}. Corollary~\ref{rem:tp-init-n} provides a simpler sufficient integrability condition~\eqref{ineq:p-finite-n}, which also allows us to deduce invariance when $\psi(\d{y})=\Psi(y)U(\d{y})$.
Compensated kernels can also be useful in verifying the
sub-median version of
 Dynkin's formula in specific examples. We establish this formula for power functions in Examples~\ref{ex:4.3} and~\ref{ex:4}.
 
The remainder of the paper is organized as follows.
Section~\ref{sec:DD} develops the basic properties of Dynkin's formula and Duhamel's formula and illustrates them with examples.
Section~\ref{sec:D-Pos} studies positive perturbations, first through Duhamel's formula and then through the perturbation series.
Section~\ref{sec:D-Neg} establishes the corresponding results for negative perturbations and presents constructions of functions satisfying the sub-median version of Dynkin's formula.

\subsection*{Notation and basic definitions}

To avoid repeatedly excluding the value $\infty$, we shall work with extended non-negative functions and adopt the following standard conventions \cite[page 18]{MR924157} that $a+\infty=\infty+a=\infty$ if $0\leq a \leq \infty$, and
\[
a \cdot \infty=\infty \cdot a =
\begin{cases}
\infty \qquad &\mbox{if } 0<a\leq \infty, \\
0 &\mbox{if } a=0.
\end{cases}
\]
Then,  for a measure space $(X,\mathcal{M}, m)$ with a non-negative measure $m$,  an $\mathcal{M}$-measurable function $f: X \to [0,\infty]$ and 
a constant $0 \leq c\leq\infty$,
\[
c\int_X f(x)\, m(\d x) = \int_X c f(x)\, m(\d x),
\]
cf. \cite[page 20]{MR924157}.
Throughout the paper, we will frequently interchange  the order of integration, as guaranteed by the conclusion of Tonelli's theorem. Accordingly, we  shall impose suitable assumptions.
We say that a measure $m$ is \textit{s-finite} if there exists a sequence $(m_k: k\in\mathbb{N})$ of finite non-negative measures on $(X,\mathcal{M})$ such that for every $A\in\mathcal{M}$, the sequence $(m_k(A): k\in\mathbb{N})$ is non-decreasing, 
and
\[
m(A)=\lim_{k\to \infty} m_k(A).
\]
Equivalently, $m$ is a countable sum of finite non-negative measures, see \cite{MR3642325}.
Since the conclusion of Tonelli's theorem holds for such measures, we shall henceforth assume that all  measures under consideration are \textit{s-finite} and, in particular, \textit{non-negative}.

\begin{remark}
Clearly, every non-negative $\sigma$-finite measure is s-finite. 
If $g : X\to [0,\infty]$ is $\mathcal{M}$-measurable and $\lambda$ is s-finite, then 
\[ 
m(\d x)= g(x) \lambda (\d x)
\]
is also s-finite. Indeed, 
if $m_k(A)=\int_A \min (g(x),k) \lambda_k(\d x)$, where $\lambda_k$ are finite and increase to $\lambda$, then $m_k$ are finite and increase to $m$.
\end{remark}

\begin{remark}
For a non-negative measurable function $f$ and s-finite measures $\nu$ and $\mu$, whenever at least one of the integrals on the left-hand side is finite, we use the abbreviation
\[
\int_X f(x)\,\nu(\d x)
-
\int_X f(x)\,\mu(\d x)
=
\int_X f(x)\,[\nu(\mathrm{d}x)-\mu(\d x)].
\]
\end{remark}

Let $(X,\mathcal{M}, m)$ be a measure space, and let
 $\mathcal{B}_{(0,\infty)}$ denote the $\sigma$-algebra of Borel subsets of $(0,\infty)$. 
Throughout the paper, 
\[
\begin{aligned}
p   &:(0,\infty)\times X\times X\to[0,\infty],\\
\tp &:(0,\infty)\times X\times X\to[0,\infty]
\end{aligned}
\]
are $\mathcal{B}_{(0,\infty)}\times \mathcal{M}\times \mathcal{M}$-measurable functions.
Following the literature,
we introduce several definitions.

\begin{definition}
\rm
We say that an $\mathcal{M}$-measurable function $h: X \to [0,\infty]$ is \textit{super-median} for $p$ if 
\[
\int_X p(t,x,y) h(y)\, m(\d y) \le h(x), \qquad t>0,\; x \in X,
\]
\textit{sub-median} for $p$ if
 \[
\int_X p(t,x,y) h(y)\, m(\d y) \ge h(x), \qquad t>0,\; x \in X,
\]
and \textit{invariant} for $p$ if
\[
\int_X p(t,x,y) h(y) \,m(\d y) = h(x), \qquad t>0,\; x \in X.
\]
\end{definition}

\begin{definition}
\rm
We say that 
an $\mathcal{M}$-measurable
function $\Phi: X \to [0,\infty]$ satisfies  (super-median)  \textit{Dynkin's formula} for $p$ and
$\phi$ provided
$\phi$ is a measure on $(X,\mathcal{M})$, and
\begin{align}\label{eq:init}
\int_X p(t,x,y)\Phi(y)\, m(\d y) + \int_0^t \int_X p(s,x,y)\, \phi(\d y) \d s =\Phi(x), \qquad t>0,\; x\in X.
\end{align}
\end{definition}

\begin{definition}
\rm
We say that
$\tp$  satisfies \textit{Duhamel's formula}
for $p$ and $V$ provided
$V$ is a measure on $(X,\mathcal M)$, and
\begin{align}\label{eq:tp-Duh}
\tp(t,x,y)
=
p(t,x,y)
+
\int_0^t\int_X
\tp(s,x,z)p(t-s,z,y)\,
V(\d z)\d s,
\qquad
t>0,\;x,y\in X.
\end{align}
Whenever this holds, we say that
$\tp$ is a Schr\"odinger perturbation of $p$ by $V$.
\end{definition}

Depending on the choice of $V$, the solution $\tp$ to \eqref{eq:tp-Duh} may or may not be comparable to $p$. In the latter situation, super-median and invariant functions for $\tp$ provide valuable tools for estimating~$\tp$.

\begin{definition}
\rm
We say that $p$ satisfies \textit{the Chapman--Kolmogorov equations} if
\begin{align}\label{eq:p-C-K}
\int_X p(s,x,z)p(t-s,z,y)\,m(\d z)
=p(t,x,y),
\qquad
0<s<t,\;x,y\in X.
\end{align}
In this case, $p$ is called a (time-homogeneous) \textit{transition density}.
\end{definition}

We end this section with a comment on Duhamel's formula.
We say that $p$ is symmetric if $p(t,x,y)=p(t,y,x)$ for all
$t>0$ and $x,y\in X$.
If both $p$ and $\tp$ are symmetric, then exchanging the spatial
variables and changing the time variable shows that \eqref{eq:tp-Duh} is equivalent to
\begin{align*}
\tp(t,x,y)
=
p(t,x,y)
+
\int_0^t\int_X
p(s,x,z)\tp(t-s,z,y)\,
V(\d z)\d s,
\qquad
t>0,\;x,y\in X.
\end{align*}
The perturbation series introduced in Section~\ref{subsec:DF-2} satisfies both forms,
without any symmetry assumption.
The general considerations concerning Duhamel's formula in Section~\ref{subsec:DF-2}
can also be adapted to this alternative form.
However, the arguments in Section~\ref{sec:D-Pos} are tailored to \eqref{eq:tp-Duh},
where the order of the kernels allows us to apply Dynkin's formula
directly to the unperturbed kernel $p$.

\section{Dynkin and Duhamel}\label{sec:DD}

\subsection{Dynkin's formula}\label{subsec:DF}

If $\Phi$ satisfies Dynkin's formula, then it is
super-median. However, the converse fails in general,
as shown in Example~\ref{ex:supermedian-boundary}.
We first provide a general method of constructing $\Phi$ and $\phi$ satisfying \eqref{eq:init}.

\begin{definition}
We say that $\Theta$ is \textit{the zero-resolvent} for $p$ and $\theta$  provided that $\theta$ is a measure on $(X,\mathcal{M})$, and
\begin{align}\label{eq:T_t}
\Theta(x) = \int_0^\infty\int_X p(t,x,y) \, \theta(\d{y})\d t, \qquad x \in X.
\end{align}
\end{definition}

\begin{lemma}\label{lem:A1}
Assume that
$\Theta$ is the zero-resolvent for $p$ and $\theta$. If $p$ is a transition density, then $\Theta$ satisfies Dynkin's formula for $p$ and $\theta$.
\end{lemma}
\begin{proof}
First we write
\begin{align*}
\Theta(x)
&= \int_0^t\int_X p(s,x,y)\, \theta(\d y) \d s 
+\int_t^\infty\int_X p(s,x,y)\, \theta(\d y) \d s.
\end{align*}
Next, by the Chapman--Kolmogorov equations \eqref{eq:p-C-K},
\begin{align*}
\int_t^\infty\int_X p(s,x,y)\, \theta(\d y) \d s
&= \int_0^\infty \int_X p(t+u,x,y) \, \theta(\d y) \d u\\
&= \int_X p(t,x,z)  \int_0^\infty \int_X p(u,z,y)\, \theta(\d y) \d u \,m(\d z)\\
&= \int_X p(t,x,z)\Theta(z)\, m(\d z).
\end{align*}
\end{proof}

Following \cite{MR3460023},
for  a measure
$\mu$ on $(X,\mathcal{M})$ and
a non-decreasing function
$f : (0,\infty)\ \to [0,\infty)$,
we consider 
\begin{align}\label{eq:hbyf}
h(x) = \int_0^\infty\int_X  f(t)p(t,x,y) \, \mu(\d{y})\d{t}, \qquad x \in X.
\end{align}

Below, $f(b^+)=\lim_{x\to b^+} f(x)$ for $b\geq 0$.
We show how \eqref{eq:hbyf} reduces to \eqref{eq:T_t}.

\begin{lemma}\label{ex:1}
Let $\mu$ be a measure on $(X,\mathcal{M})$ and
let $f : (0,\infty)\ \to [0,\infty)$ be non-decreasing.
Denote by $\lambda$ the unique
Borel measure on $(0,\infty)$
such that
\[
f(a^+)-f(0^+)= \lambda((0,a]), \qquad a > 0.
\]
Let $h$ be given by \eqref{eq:hbyf}.
If $p$ is a transition density,
then
$h$ is the zero-resolvent for $p$ and 
$\theta$, where
\[
\theta(\d x)=\bar{h}(x)m(\d x) + f(0^+) \mu(\d x),
\]
and
\[
\bar{h}(x)=\int_{(0,\infty)} \int_X p(t,x,y) \, \mu(\d y) \lambda (\d t), \qquad x\in X.
\]
\end{lemma}
\begin{proof}
Since $f$ is non-decreasing,
\begin{align*}
h(x)&=\int_0^\infty \int_X f(t^+) p(t,x,y) \, \mu(\d y) \d t\\
&=\int_0^\infty \int_X \int_{(0,\infty)} \mathds{1}_{(0,t]}(s) \, \lambda (\d s)\,  p(t,x,y)  \, \mu(\d y) \d t +
\int_0^\infty \int_X p(t,x,y) \, f(0^+)\mu(\d y) \d t
\end{align*}
and, by the Chapman--Kolmogorov equations \eqref{eq:p-C-K}, the first term equals
\begin{align*}
&\int_{(0,\infty)} \int_0^\infty \int_X  \mathds{1}_{[s,\infty)}(t) \,  p(t,x,y)  \, \mu(\d y) \d t \, \lambda (\d s)\\
&=\int_{(0,\infty)}  \int_{-s}^\infty  \int_X  \mathds{1}_{[s,\infty)}(s+u) p(s+u,x,y) \,  \mu(\d y) \d u \, \lambda (\d s)\\
&=\int_{(0,\infty)}  \int_0^\infty  \int_X   \int_X p(u,x,z) p(s,z,y)\, m(\d z)  \mu(\d y) \d u \, \lambda(\d s)\\
&=\int_0^\infty  \int_X p(u,x,z) \left( \int_{(0,\infty)} \int_X  p(s,z,y)\,   \mu(\d y)\, \lambda (\d s) \right)  m(\d z)\d u\\
&=\int_0^\infty \int_X p(u,x,z) \bar{h}(z)\, m(\d z) \d u.
\end{align*}
\end{proof}

We provide several examples of zero-resolvents. 
 
\begin{example}\label{ex:5}
Let $d\in\mathbb{N}$  satisfy $d\geq 3$ and $g(t,x,y)=g(t,y-x)= (4\pi t)^{-d/2}e^{-\frac{|y-x|^2}{4t}}$, $t>0$, $x,y\in\Rd$, be the Gauss-Weierstrass kernel. Then, for all  $\beta\in (0,d-2)$  and $x\in \Rd$,
\[
 |x|^{-\beta} = \int_0^\infty \int_{\Rd}g(t,x,y) \,\beta (d-2-\beta)|y|^{-\beta-2} \, \d y  \d t.
\]
To see this, we recall from \cite[(24)]{MR3460023} that, for $-\infty<\gamma<d/2-1$,
\begin{align}\label{eq:gauss-1}
\int_0^\infty s^\gamma g(s,x)  \, \d s = \frac{\Gamma(d/2-1-\gamma)}{4^{\gamma+1}\pi^{d/2}}  |x|^{2\gamma-d+2}.
\end{align}
Thus,
\[
\int_0^\infty \int_{\Rd} t^\gamma g(t,x,y) \, \delta_0(\d y) \d t = \frac{\Gamma(d/2-1-\gamma)}{4^{\gamma+1}\pi^{d/2}}  |x|^{2\gamma-d+2}
\]
and
\[
\int_0^\infty \int_{\Rd} \gamma t^{\gamma-1} g(t,x,y) \, \delta_0(\d y) \d t = \frac{\gamma \Gamma(d/2-\gamma)}{4^{\gamma}\pi^{d/2}}  |x|^{2\gamma-d}.
\]
Now, we apply
Lemma~\ref{ex:1} to $p=g$, $\mu=\delta_0$, $f(t)=t^\gamma$,  and $\gamma=d/2-1-\beta/2$, and use the above explicit expression.
\end{example}

\begin{example}\label{Ex:gausshalfline}
Let $g_D(t,x,y) = \frac{1}{\sqrt{4\pi t}} \left(e^{-\frac{(x-y)^2}{4t}}- e^{-\frac{(x+y)^2}{4t}}\right)$, $t>0$, $x,y\in D$, be the Dirichlet heat kernel of $\Delta$ on $D=(0,\infty)$. 
Then, for all $\beta\in (0,1)$ and $x\in D$,
\[
 x^{1-\beta} = \int_0^\infty \int_D g_D(t,x,y) \beta(1-\beta) y^{-1-\beta} \, \d{y} \d{t}.
\]
Equivalently, in the Doob conditioned setting with $m(\d x)=x^2 \d x$,
\[
 x^{-\beta} = \int_0^\infty \int_D \frac{g_D(t,x,y)}{x \, y} \beta(1-\beta) y^{-\beta-2} \, m(\d{y}) \d{t}.
\]
Indeed, the Green function is $G_D(x,y) = \int_0^\infty g_D(t,x,y) \d{t} = x \land y$. Thus,
\begin{align*}
\int_0^\infty \int_D g_D(t,x,y) y^{-1-\beta} \, \d{y} \d{t} 
	&= \int_D G_D(x,y) y^{-1-\beta} \, \d{y} \\
	&=\int_0^x y^{-\beta} \, \d{y} + x\int_x^\infty y^{-1-\beta} \, \d{y} 
	= \frac{1}{\beta(1-\beta)} x^{1-\beta}.
\end{align*}
Fortunately, the Green function admits a simple explicit formula. A direct adaptation of the approach used in Example~\ref{ex:5}, however, may lead to lengthy and rather tedious calculations. The argument would be as follows.
For any $\gamma \in (0, \frac{1-\beta}{2})$, take
$f(t) = t^\gamma$ and $\mu(\d{y}) = y^{-\beta - 2 \gamma - 1}\d{y}$ in \eqref{eq:hbyf}, and then apply Lemma~\ref{ex:1}. 
First, integrating by parts, we get for $\gamma<1/2$ such that $\gamma\neq -1/2$,
\begin{align*}
\int_0^\infty s^\gamma g_D(s,x,y) \,\d{s} = \frac{\Gamma\left(\frac{1}{2}-\gamma\right)}{4^{\gamma+1}\sqrt{\pi}(\gamma+\frac12)} \left( |x+y|^{2\gamma+1} - |x-y|^{2\gamma+1} \right).
\end{align*}
Next, using the substitution $y =x z$, we obtain
\[
\int_0^\infty \int_D t^\gamma g_D(t,x,y)  y^{-\beta-2\gamma-1} \d{y} \d{t} = c_{\beta,\gamma} \, x^{1-\beta}
\]
and
\[
\int_0^\infty \int_D \gamma t^{\gamma-1} g_D(t,x,y)  y^{-\beta-2\gamma-1} \d{y} \d{t} = \bar{c}_{\beta,\gamma} \, x^{-1-\beta}.
\]
For $\gamma \in (0, \frac{1-\beta}{2})$ both constants are positive and finite.
It  can be shown that the constant 
$\bar{c}_{\beta,\gamma}/c_{\beta,\gamma}$ 
does not depend on $\gamma$ and is equal to $\beta(1-\beta)$, although a direct verification of the equality is not entirely straightforward.
In the present setting, however, both constants can be computed explicitly by adapting the approach from Example~\ref{Ex:gausshalfspace}.
\end{example}

\begin{example}\label{Ex:gausshalflinev2}
As in Example \ref{Ex:gausshalfline}, let $g_D(t,x,y)$ be the Dirichlet heat kernel of $\Delta$ on $D=(0,\infty)$.
Then, for every $x\in D$,
\[
x\ln\left(\frac{x+1}{x}\right)= \int_0^\infty \int_D g_D(t,x,y) \frac1{y(y+1)^2}\, \d y \d t.
\]
Equivalently, for every $x\in D$,
\[
\ln\left(\frac{x+1}{x}\right)= \int_0^\infty \int_D \frac{g_D(t,x,y)}{x\, y} \frac1{y^2(y+1)^2}\, m(\d y) \d t.
\]
Indeed,
\begin{align*}
\int_0^\infty \int_D g_D(t,x,y) \frac1{y(y+1)^2}\, \d y \d t 
&= \int_D G_D(x,y) \frac{1}{y(y+1)^2}\, \d{y} \\
&= \int_0^x \frac{1}{(y+1)^2}\, \d{y} + x\int_x^\infty \frac{1}{y(y+1)^2} \, \d{y} \\
&= 1-\frac{1}{x+1}+ x \left(\ln(x+1) - \ln x - \frac{1}{x+1}\right) \\
&= x\ln\left(\frac{x+1}{x}\right).
\end{align*}
\end{example}

\begin{example}
Let $d\in\mathbb{N}$ and $\alpha\in (0,2)$ satisfy $d>\alpha$. Let $p(t,x,y)=p(t,y-x)$, $t>0$, $x,y\in\Rd$, be the heat kernel of $-(-\Delta)^{\alpha/2}$.
Then, for all  $\beta\in (0,d-\alpha)$ and $x\in\Rd$,
\[
 |x|^{-\beta} = \int_0^\infty\int_{\Rd} p(s,x,y) \kappa_\beta |y|^{-\beta-\alpha}\, \d y \d s.
\]
We recall from \cite[p. 236]{MR3460023} that, for $-1< \gamma <(d-\alpha)/\alpha$,
\[
\int_0^\infty s^\gamma p(s,x)\, \d s = 
\frac{\Gamma(\gamma+1)\Gamma(\frac{d-\alpha(\gamma+1)}{2})}{\Gamma(\frac{\alpha(\gamma+1)}{2}) 4^{\frac{\alpha(\gamma+1)}{2}} \pi^{d/2}} |x|^{\alpha(\gamma+1)-d}.
\]
Thus, for $0< \gamma <(d-\alpha)/\alpha$,
\[
\int_0^\infty \int_{\Rd} t^\gamma p(t,x,y) \, \delta_0(\d y) \d t = c_{\gamma}  |x|^{\alpha\gamma-(d-\alpha)}
\]
and
\[
\int_0^\infty \int_{\Rd} \gamma t^{\gamma-1} p(t,x,y) \, \delta_0(\d y) \d t = \bar{c}_{\gamma}  |x|^{\alpha \gamma-d}.
\]
The explicit constants $c_{\gamma}$ and $\bar{c}_{\gamma}$ satisfy $\bar{c}_{\gamma}/c_{\gamma}=\kappa_{(d-\alpha)-  \alpha\gamma}$, with $\kappa_\beta$ defined in the Introduction.
Now, we apply
Lemma~\ref{ex:1} to $p$, $\mu=\delta_0$, $f(t)=t^\gamma$,  and $\gamma=(d-\alpha)/\alpha-\beta/\alpha$, and use the above explicit expression.
\end{example}

For more examples involving non-local operators in the whole space $\Rd$, see \cite[Section~2.3]{MR4705624} and \cite[Section~3]{MR4746984}, where $p$ is the heat kernel for the fractional relativistic operator and a fractional power of the Bessel operator, respectively.

\begin{example}\label{Ex:stablehalfline}
Let $\alpha\in (0,2)$ and $p_D(t,x,y)$, $t>0$, $x,y\in D$, be the Dirichlet heat kernel of $-(-\Delta)^{\alpha/2}$ on $D=(0,\infty)$. Then, see \cite[Theorem 1.2]{MR4660833}, for all $\beta\in (0,1)$ and $x\in D$,
\begin{align*}
x^{\alpha/2-\beta}
=\int_0^\infty \int_D p_D(t,x,y) \frac{\Gamma(\beta+\alpha/2)\Gamma(1-\beta+\alpha/2)}{\Gamma(\beta)\Gamma(1-\beta)}  y^{-\alpha/2-\beta}\,  \d{y} \d{t}.
\end{align*}
In the Doob conditioned setting with $m(\d x)=x^\alpha \d x$,
\begin{align*}
x^{-\beta}
=\int_0^\infty \int_D \frac{p_D(t,x,y)}{x^{\alpha/2}y^{\alpha/2}} \frac{\Gamma(\beta+\alpha/2)\Gamma(1-\beta+\alpha/2)}{\Gamma(\beta)\Gamma(1-\beta)}  y^{-\beta-\alpha}\,  m(\d{y}) \d{t}.
\end{align*}
Note that, for $\alpha=2$, the value of the above constant agrees with that in Example~\ref{Ex:gausshalfline}. Here, the  approach through formula~\eqref{eq:hbyf}  goes by  taking, for any $\gamma\in (0,\frac{1-\beta}{\alpha})$, $f(t) = t^\gamma$ and $\mu(\mathrm{d}y) = y^{-\beta-\alpha \gamma-\alpha/2} \, \mathrm{d}y$;
\[
\int_0^\infty \int_D t^\gamma p_D(t,x,y)  y^{-\beta-\alpha\gamma-\alpha/2} \d{y} \d{t} = c_{\beta,\gamma} \, x^{\alpha/2-\beta}
\]
and
\[
\int_0^\infty \int_D \gamma t^{\gamma-1} p_D(t,x,y)  y^{-\beta-\alpha\gamma-\alpha/2} \d{y} \d{t} = \bar{c}_{\beta,\gamma} \, x^{-\alpha/2-\beta}.
\]
Obtaining explicit expressions for $c_{\beta,\gamma}$ and $\bar{c}_{\beta,\gamma}$ appears to be a rather delicate problem.
The approach based on the relation \eqref{eq:T_t}, now formulated in general in Lemma~\ref{ex:1}, was first introduced in \cite{MR4660833}  as a starting point for analysing the ratio $\bar{c}_{\beta,\gamma}/c_{\beta,\gamma}$. The explicit evaluation of this ratio, which relies on lengthy and technically involved calculations using the Ikeda--Watanabe formula, is the main result of \cite{MR4660833}.
A semi-explicit formula for the Green function is also available, see \cite[(30)]{MR3888401}, but applying it as in Example~\ref{Ex:gausshalfline} would involve substantial technical work.
\end{example}

\begin{example}[Generalization of Example~\ref{Ex:gausshalfline}]\label{Ex:gausshalfspace}
Let $d\geq 1$ and $H = \{x= (\tilde{x},x_d) \in \Rd : x_d>0\}$ be a half-space. We consider $g_H(t,x,y)$, $t>0$, $x,y\in H$, the transition density of Brownian motion killed on exiting $H$. Then,
for all $\beta\in(0,d)$ and $x\in H$,
\[
x_d |x|^{-\beta} =\int_0^\infty \int_H g_H(t,x,y)  \beta(d-\beta)  y_d |y|^{-\beta-2} \, \d{y} \d{t}.
\]
In the Doob conditioned setting with $m(\d x)=x_d^2 \d x$,
\[
|x|^{-\beta}=\int_0^\infty \int_H \frac{g_H(t,x,y)}{x_d \, y_d}  \beta(d-\beta) |y|^{-\beta-2} m(\d{y})\d{t}.
\]
To see this, first recall that
\[
g_H(t,x,y)=g(t,x-y)-g(t,x-y^*),
\qquad
y^*=(\widetilde y,-y_d),
\]
where  $g(t,x)= (4\pi t)^{-d/2}e^{-\frac{|x|^2}{4t}}$.
By symmetry, $g(s,y^*)=g(s,y)$, and hence
\begin{align*}
\int_H g_H(t,x,y)y_dg(s,y) \, \d{y}
=\int_{y_d>0}g(t,x-y)y_dg(s,y)\,\d{y} -\int_{y_d>0}g(t,x-y^*)y_dg(s,y)\,\d{y}.
\end{align*}
By substitution in the latter integral,
\[
\int_H g_H(t,x,y)y_dg(s,y)\,\d{y}
=
\int_{\Rd}g(t,x-y)y_dg(s,y)\,\d{y}.
\]
Since, see also \cite[page 465]{MR1978999},
\begin{align}\label{eq:4g}
g(t,x-y)g(s,y)
=
g(t+s,x)\,
g\left(\frac{st}{s+t},
y-\frac{s}{s+t}x\right),
\end{align}
we get
\begin{align*}
\int_{\Rd}g(t,x-y)y_dg(s,y)\,\d{y}
&=g(t+s,x)
  \int_{\Rd}y_d\,
  g\left(\frac{st}{s+t},
  y-\frac{s}{s+t}x\right)\d{y}\\
&=g(t+s,x)
  \int_{\Rd}
  \left(z_d+\frac{s}{s+t}x_d\right)
  g\left(\frac{st}{s+t},z\right)\d{z}\\
&=\frac{s x_d}{s+t}g(t+s,x).
\end{align*}
Consequently, by Tonelli's theorem, the substitution $u=s+t$ and \eqref{eq:gauss-1},
\begin{align*}
\int_0^\infty\int_H
g_H(t,x,y)y_d|y|^{-\beta-2}\d{y}\d{t}
&= \frac{2^{d-\beta-2}\pi^{d/2}}{\Gamma(\beta/2+1)} x_d
\int_0^\infty\int_0^\infty
\frac{s^{(d-\beta-2)/2}}{s+t}
g(s+t,x) \, \d{s}\d{t}\\
&=\frac{2^{d-\beta-2}\pi^{d/2}}{\Gamma(\beta/2+1)} \frac{2}{d-\beta}\,
x_d\int_0^\infty
u^{(d-\beta-2)/2}g(u,x)\,\d{u}\\
&= \frac{1}
{\beta(d-\beta)}\,
x_d|x|^{-\beta}.
\end{align*}
\end{example}

We continue our considerations on Dynkin's formula.

\begin{lemma}\label{lem:9}
Assume that  $\Phi$ satisfies Dynkin's formula for $p$ and $\phi$.
Then, for every function $\Omega$ invariant for~$p$, 
the function $\Phi+\Omega$ also satisfies
Dynkin's formula for $p$ and $\phi$. \end{lemma}
\begin{proof}
 Clearly, since $\Omega$ is invariant for $p$,
\[
\int_X p(t,x,y)(\Phi(y) +\Omega (y))\, m(\d{y}) + \int_0^t \int_X p(s,x,y) \phi(\d{y}) \d{s} = \Phi(x) + \Omega(x),
\qquad x\in X.
\]
\end{proof}

\begin{lemma}\label{lem:11}
Assume that $\Phi$ satisfies Dynkin's formula for $p$ and $\phi$. If $p$ is a transition density and $\Phi(x)<\infty$ for every $x\in X$, then
\[
\Phi(x)=\Theta(x)+
\Omega(x), \qquad x\in X,
\]
where $\Theta$ is the zero-resolvent for $p$ and $\phi$, and $\Omega$ is invariant for $p$ and satisfies
\begin{align}\label{eq:Omega-rep}
\Omega(x)=\lim_{t\to\infty} \int_X p(t,x,y)\Phi(y) \, m(\d{y}), \qquad x\in X.
\end{align}
\end{lemma}
\begin{proof}
By Dynkin's formula \eqref{eq:init}, the integral
$\int_X p(t,x,y)\Phi(y) \, m(\d{y})$ is finite and non-increasing  in $t>0$.
Therefore, $\Omega$ is well-defined by \eqref{eq:Omega-rep}.
By the Chapman--Kolmogorov equations
\eqref{eq:p-C-K}, 
and the dominated convergence theorem,
for $t>0$ and $x\in X$,
\begin{align*}
\Omega(x) = \lim_{s\to\infty}\int_X p(t+s,x,y)\Phi(y) \, m(\d{y})
&= \lim_{s\to\infty}\int_X \int_X p(t,x,z)p(s,z,y) \Phi(y) \, m(\d{y}) \, m(\d{z})\\
&= \int_X p(t,x,z) \left(\lim_{s\to\infty}\int_X p(s,z,y) \Phi(y) \, m(\d{y})\right) \, m(\d{z})\\
&= \int_X p(t,x,z) \Omega(z) \, m(\d{z}).
\end{align*}
This proves the invariance.
By letting $t \to \infty$ in Dynkin's formula \eqref{eq:init}, we get
\[
\Phi(x) = \lim_{t\to\infty}\int_X p(t,x,y)\Phi(y) \, m(\d{y}) + \lim_{t\to\infty}\int_0^t \int_X p(s,x,y) \phi(\d{y}) \d{s}.
\]
The first limit is an invariant function for $p$, and the second is the zero-resolvent for $p$ and $\phi$.
\end{proof}

\begin{example}\label{ex:supermedian-boundary}
We use the notation of Example~\ref{Ex:gausshalfline}.
The constant function $u\equiv 1$ on $D$ is super-median
for $g_D$, but does not satisfy Dynkin's formula
with any (s-finite) measure $\phi$.

Indeed, by the reflection principle
\cite[Chapter~III, Proposition~3.7]{MR1725357},
\[
\int_D g_D(t,x,y)u(y)\,\d y
=\frac{1}{\sqrt{\pi t}}\int_0^x e^{-z^2/(4t)}\,\d z
<1=u(x), \qquad t>0,\,x\in D.
\]
If Dynkin's formula held for some $\phi$, then
by Lemma~\ref{lem:11} and Tonelli's theorem, we would have
\[
1=u(x)=\int_D G_D(x,y)\,\phi(\d y)
=\int_0^\infty (x\wedge y)\,\phi(\d y),
\qquad x\in D.
\]
Taking $x=1$ yields
$
\int_0^\infty (1\wedge y)\,\phi(\d y)=1
$.
Hence, by the dominated convergence theorem,
\[
    1=\lim_{x\to 0^+}
    \int_0^\infty (x\wedge y)\,\phi(\d y)=0,
\]
which is a contradiction.
\end{example}

\begin{remark}
If $p$ is probabilistic, i.e., $\int_X p(t,x,y)\,m(\d y)=1$, $t>0$, $x\in X$, then $\Omega \equiv 1$ is invariant for $p$. Hence, from Lemma~\ref{lem:9} we deduce that for a given measure $\phi$ there may be more than one function $\Phi$ satisfying Dynkin's formula \eqref{eq:init}.
\end{remark}

\subsection{Duhamel's formula}\label{subsec:DF-2} 

\begin{definition}
We say that $\rho$ is the perturbation series for $p$ and $V$ provided that
$V$ is a measure on $(X,\mathcal{M})$, and
\begin{align}\label{eq:pertser}
\rho(t,x,y) =\sum_{n=0}^{\infty} p_n(t,x,y),
\qquad t>0,\; x,y\in X,
\end{align}
where
\begin{equation}\label{eq:pn}
\begin{aligned}
p_0(t,x,y)&=p(t,x,y),\\
p_n(t,x,y)&=\int_0^t\int_X p_{n-1}(s,x,z) p_0(t-s,z,y)\,V(\d z)\d s ,\qquad n\geq 1.
\end{aligned}
\end{equation}
\end{definition}

We record the following fact, adapted from the general theory developed in \cite{MR2457489} and based solely on the algebraic structure of the perturbation series and repeated changes of the order of integration.

\begin{remark}\label{rem:1}
Assume that $\rho$ is the perturbation series for $p$ and $V$. If $p$ is a transition density, then $\rho$ is also a transition density.
\end{remark}

We now discuss Duhamel's formula.

\begin{lemma}\label{lem:bp-Duhamel}
Assume that $\rho$ 
is the perturbation series
for $p$ and $V$. Then, $\rho$
satisfies Duhamel's formula for $p$ and $V$.
\end{lemma}
\begin{proof}
We have
\begin{align*} 
\rho(t,x,y) &= p(t,x,y) + \sum_{n=0}^\infty p_{n+1} (t,x,y)\\
&= p(t,x,y) + \sum_{n=0}^\infty\int_0^t \int_X p_n (s,x,z) p(t-s,z,y) \,V(\d z)\d s \\
&= p(t,x,y) + \int_0^t \int_X \rho (s,x,z) p(t-s,z,y) \,V(\d z)\d s.
\end{align*}
\end{proof}

\begin{lemma}\label{lem:D+w=D}
Assume that $\tp$ satisfies Duhamel's formula for $p$ and $V$.
Then, for every $\mathcal{B}_{(0,\infty)}\times \mathcal{M}\times \mathcal{M}$-measurable function
$\omega: (0,\infty)\times X \times X \to [0,\infty]$ such that
\begin{align}\label{eq:omega-inv}
\omega(t,x,y)=\int_0^t\int_X \omega (s,x,z) p(t-s,z,y)\, V(\d z) \d s,
\qquad t>0,\; x,y\in X,
\end{align}
the function
$\tp + \omega$ also satisfies
Duhamel's formula for $p$ and $V$.
\end{lemma}
\begin{proof}
Clearly, by \eqref{eq:omega-inv},
\begin{align*}
\tp(t,x,y)+\omega(t,x,y)
=
p(t,x,y)
+
\int_0^t\int_X
(\tp(s,x,z) +\omega(s,x,z)) p(t-s,z,y)\,
V(\d z)\d s
\end{align*}
for all $t>0$, $x,y\in X$.
\end{proof}

\begin{remark}\label{rem:psi-inv}
If a measurable function $\psi\geq 0$ satisfies
\begin{align}\label{eq:psi-inv}
\psi(t,y)=\int_0^t\int_X \psi (s,z) p(t-s,z,y)\, V(\d z) \d s,
\qquad t>0,\; y\in X,
\end{align}
then for any measurable function $f \ge 0$, $\omega(t,x,y) = f(x)\psi(t,y)$ satisfies \eqref{eq:omega-inv}. In particular, we can take $f \equiv1$. Examples of such functions $\psi$ will be given below.
We would also like to point out a possible method of constructing $\psi$: given Duhamel's formula \eqref{eq:tp-Duh}, one could divide both sides by a suitably chosen function $f$ such that $\lim_{x\to a} \tp(t,x,y)/f(x)$ exists and is non-trivial, while $\lim_{x\to a} p(t,x,y)/f(x)=0$; treating the integral term of \eqref{eq:tp-Duh} might, however, be cumbersome. See Remark~\ref{rem:4} for further discussion.
\end{remark}

\begin{lemma}\label{lem:13}
Assume that $\omega$ is measurable, non-negative and satisfies \eqref{eq:omega-inv}. Then, \eqref{eq:omega-inv} is also satisfied by 
\[
\omega_1(t,x,y)= \int_0^t \omega(t-s,x,y)\, \d{s},
\qquad
t>0,\; x,y\in X.
\]
\end{lemma}
\begin{proof}
By Tonelli's theorem,
\begin{align*}
&\int_0^t\int_X \omega_1 (s,x,z) p(t-s,z,y)\, V(\d z) \d s \\
&= \int_0^t\int_X \int_0^s\omega (s-r,x,z) p(t-s,z,y) \, \d{r} \,V(\d z) \d s \\
&= \int_0^t \int_r^t \int_X \omega (s-r,x,z) p(t-s,z,y) \,V(\d z) \d{s} \,  \d r\\
&= \int_0^t\int_0^{t-r} \int_X  \omega (s,x,z) p(t-r-s,z,y) \, V(\d z)  \d{s} \, \d r = \int_0^t \omega(t-r,x,y)\,\d{r} = \omega_1(t,x,y).
\end{align*}
\end{proof}

\begin{lemma}\label{lem:12}
Assume that $\tp$ satisfies Duhamel's formula for $p$ and $V$.
Then, for every $N\in \mathbb{N}\cup\{0\}$, 
\begin{align}\label{eq:tp-Duh-it}
\tp(t,x,y) =\sum_{n=0}^{N}p_n(t,x,y)
+\int_0^t \int_X \tp (s,x,z) p_N(t-s,z,y) \,V(\d z)\d s,
\qquad
t>0,\;x,y\in X.
\end{align} 
\end{lemma}
\begin{proof}
It suffices to iterate Duhamel's formula \eqref{eq:tp-Duh} and change the order of integration to obtain the recursive formulae. 
Indeed, for $N=0$ we have \eqref{eq:tp-Duh}. Now, suppose that \eqref{eq:tp-Duh-it} holds for $N\in \mathbb{N}\cup\{0\}$. Then, by Tonelli's theorem,
\begin{align*}
&\tp(t,x,y) = p(t,x,y) +\int_0^t \int_X \tp (s,x,z) p(t-s,z,y) \,V(\d z)\d s \\
&= p(t,x,y)
 +\int_0^t \int_X
   \Big[\sum_{n=0}^{N}p_n(s,x,z)  \\
&\phantom{{}= p(t,x,y)
 +\int_0^t \int_X \Bigl[}
   +\int_0^s \int_X \tp(r,x,w) p_N(s-r,w,z)\,V(\d w)\d r
   \Big]
   p(t-s,z,y)\,V(\d z)\d s \\
&= \sum_{n=0}^{N+1}p_n(t,x,y) +\int_0^t \int_X  \tp (r,x,w) \int_r^t \int_X p_N(s-r,w,z) p(t-s,z,y) \,V(\d z)\d s \,V(\d w)\d r \\
&= \sum_{n=0}^{N+1}p_n(t,x,y) +\int_0^t \int_X  \tp (r,x,w) p_{N+1}(t-r,w,y) \,V(\d w)\d r.
\end{align*}
Thus, the result follows by induction.
\end{proof}

\begin{lemma}\label{lem:omega4}
Assume that $\tp$ satisfies Duhamel's formula for $p$ and $V$.
If $\tp(t,x,y)<\infty$ for all $t>0$, $x,y\in X$, then 
\[
\tp(t,x,y) =\rho(t,x,y) +\omega(t,x,y),
\qquad
t>0,\; x,y\in X,
\]
where $\rho$ is the perturbation series for $p$ and $V$, and $\omega$ is measurable, non-negative, satisfies \eqref{eq:omega-inv} and 
\begin{align}\label{eq:omega-repr}
\omega(t,x,y)=\lim_{N\to \infty} \int_0^t \int_X \tp (s,x,z) p_N(t-s,z,y) \,V(\d z)\d s,
\qquad
t>0,\; x,y\in X.
\end{align}
\end{lemma}
\begin{proof}
By \eqref{eq:tp-Duh-it}, the integral
$\int_0^t \int_X \tp (s,x,z) p_N(t-s,z,y) \,V(\d z)\d s$ is finite and non-increasing  in $N\in\mathbb{N}$.
Therefore, $\omega$ is well-defined by \eqref{eq:omega-repr}. By the dominated convergence theorem,
\begin{align*}
&\int_0^t\int_X \omega (s,x,z) p(t-s,z,y)\, V(\d z) \d s\\
&= \int_0^t\int_X \lim_{N\to\infty} \left( \int_0^s \int_X \tp (u,x,w) p_N(s-u,w,z) \,V(\d w)\d u \right) p(t-s,z,y)\, V(\d z) \d s\\
&= \lim_{N\to\infty}  \int_0^t\int_X  \tp (u,x,w) \left( \int_u^t \int_X p_N(s-u,w,z) p(t-s,z,y)\, V(\d z) \d s \right) V(\d w)\d u\\
&= \lim_{N\to\infty}  \int_0^t\int_X  \tp (u,x,w) p_{N+1}(t-u,w,y) V(\d w)\d u
= \omega(t,x,y).
\end{align*}
This proves \eqref{eq:omega-inv}.
The rest of the result follows by letting $N \to \infty$ in \eqref{eq:tp-Duh-it}.
\end{proof}

Below we present several examples of functions $\omega$ and $\psi$ satisfying \eqref{eq:omega-inv} and \eqref{eq:psi-inv}, respectively. Examples~\ref{ex:6}--\ref{ex:8} concern functions that are either directly derived from or closely related to the so-called self-similar solutions. We provide further comments on this connection in Remark~\ref{rem:4}. The function in Example~\ref{ex:9} fits into a similar framework, whereas the function appearing in Corollary~\ref{cor:2} is of a different nature.

\begin{example}\label{ex:6}
Let $d\in\mathbb{N}$ satisfy $d\geq 3$ and $g(t,x,y)=g(t,y-x)= (4\pi t)^{-d/2}e^{-\frac{|y-x|^2}{4t}}$, $t>0$,~$x,y\in\Rd$. 
Then, for every $\delta\in(0,d-2)$,
the function $\psi(t,y)=t^\delta |y|^{-\delta}g(t,y)$, $t>0$, $y\in \Rd\setminus\{0\}$, satisfies  \eqref{eq:psi-inv} for $p=g$ and
$V(\d x)= \delta(d-2-\delta)|x|^{-2}\d x$; see also \cite[Section~10.4]{MR1760280}.

Indeed, by \eqref{eq:4g} we get
\begin{align*}
\int_0^t\int_{\Rd} \psi (s,z) &g(t-s,z,y)\, V(\d z) \d s\\
&=\delta(d-2-\delta)g(t,y)
\int_0^t\int_{\Rd}
s^\delta g\left(\frac{s(t-s)}{t},z-\frac{s}{t}y\right)
|z|^{-\delta-2}\,\d z\d s .
\end{align*}
It remains to show that
\[
\delta(d-2-\delta)
\int_0^t\int_{\Rd}
s^\delta g\left(\frac{s(t-s)}{t},z-\frac{s}{t}y\right)
|z|^{-\delta-2}\,\d z \d s
=
t^\delta |y|^{-\delta}.
\]
By  the scaling $g(t,x)=s^d g(s^2t,sx)$, $s>0$, it suffices to consider $t=1$. Using the change of variables $z=sw$ and the scaling, we obtain
\[
\int_0^1\int_{\Rd}
s^\delta g(s(1-s),z-sy)|z|^{-\delta-2}\,\d z\d s
= \int_0^1\int_{\Rd}
s^{-2}g\left(\frac{1-s}{s},w-y\right)|w|^{-\delta-2}\,\d w \d s .
\]
Now, substituting $r=(1-s)/s$, by Example \ref{ex:5}, we get
\[
\int_0^1\int_{\Rd}
s^\delta g(s(1-s),z-sy)|z|^{-\delta-2}\,\d z \d s
= \int_0^\infty\int_{\Rd} g(r,w-y)|w|^{-\delta-2}\,\d w \d r
= \frac{|y|^{-\delta}}{\delta(d-2-\delta)}.
\]
\end{example}

\begin{example}\label{ex:7}
Let $g_D(t,x,y) = \frac{1}{\sqrt{4\pi t}} \left(e^{-\frac{(x-y)^2}{4t}}- e^{-\frac{(x+y)^2}{4t}}\right)$, $t>0$, $x,y\in D$, and $D=(0,\infty)$. Then, for every $\delta\in (0,1)$, the function $\psi(t,y)=t^{\delta-1}y^{1-\delta}g(t,y)= \frac{y^{1-\delta}t^{\delta-1}}{\sqrt{4\pi t}} e^{-\frac{y^2}{4t}}$, $t>0$, $y\in D$,
satisfies  \eqref{eq:psi-inv} for $p=g_D$ and $V(\d x) = \delta(1-\delta) x^{-2}\d x$. 

Indeed, the steps are similar to those in Example~\ref{ex:6}.
We use the identity \eqref{eq:4g} twice to obtain
\begin{align*}
\int_0^t\int_D \psi (s,z) &g_D(t-s,z,y)\, V(\d z) \d s\\
&=\delta(1-\delta)g(t,y)
\int_0^t\int_D
s^{\delta-1} g_D\left(\frac{s(t-s)}{t},z,\frac{s}{t}y\right)
z^{-\delta-1}\,\d z\d s .
\end{align*}
Then, by the scaling $g_D(t,x,y) = sg_D(s^2t,sx,sy)$, $s>0$, we reduce the problem to the case $t=1$. We make
two substitutions: $z=sw$ and  $r=(1-s)/s$, and use 
Example \ref{Ex:gausshalfline} to get
\[
\int_0^1\int_D
s^{\delta-1} g_D(s(1-s),z,sy)z^{-\delta-1}\,\d z \d s
= \int_0^\infty\int_D g_D(r,w,y)w^{-\delta-1}\,\d w \d r
= \frac{y^{1-\delta}}{\delta(1-\delta)}.
\]
\end{example}

\begin{example}\label{ex:8}
Let $d\in\mathbb{N}$ and $\alpha\in (0,2)$ satisfy $d>\alpha$. Let $p(t,x,y)=p(t,y-x)$, $t>0$, $x,y\in\Rd$, be the heat kernel of $-(-\Delta)^{\alpha/2}$, and let
$\tp$ denote the perturbation series for $p$ and $V(\d x) = \kappa |x|^{-\alpha}\d x$, where $\kappa=\kappa_\delta$ for some $\delta\in (0,(d-\alpha)/2]$,  as defined in the Introduction.
Let $\psi(t,y) = \lim_{x\to0} |x|^\delta \tp(t,x,y)$, $t>0$, $y\in \Rd\setminus\{0\}$, be the so-called self-similar solution constructed in \cite{MR4595537}. 
Then, by \cite[Lemma~5.1]{MR4595537} and the symmetry of $p$ in the spatial variables, $\psi$ satisfies \eqref{eq:psi-inv} for $p$ and $V$.
\end{example}

For the one-dimensional fractional Laplacian on $(0,\infty)$ with a Hardy potential, the self-similar theory remains to be developed. Thus, we present a somewhat simpler approach, which is similar in spirit to \cite[Corollary~6.3]{MR4595537}, see also Remark~\ref{rem:psi-inv} and Lemma~\ref{lem:13}.

\begin{lemma}\label{lem:psi-construction}
Assume that $\Phi$ satisfies Dynkin's formula for $p$ and $\phi$, and that  
$\tp$ satisfies Duhamel's formula for $p$ and $V$. Suppose  that $\Phi(x)<\infty$ for every $x\in X$, that $\phi(\d x)= \Phi(x)V(\d x)$, and that $\Phi$ is super-median for $\tp$.
If $p$ and $\tp$ are symmetric, 
then
\[
\psi(t,y)=\Phi(y)-\int_X\tp(t,y,z)\Phi(z)\,m(\d z),\qquad t>0,\; y\in X,
\]
is non-negative and satisfies \eqref{eq:psi-inv}.
\end{lemma}
\begin{proof}
The non-negativity follows from the super-median property of $\Phi$. 
Since $\Phi$ is finite, all terms in Dynkin's formula are finite. The symmetry gives
\[
\int_0^t\int_X p(t-s,z,y)\Phi(z)\,V(\d{z}) \d{s}
=\Phi(y)-\int_X p(t,y,w)\Phi(w)\,m(\d{w}).
\]
By symmetry, Tonelli's theorem, Duhamel's formula and Dynkin's formula,
\begin{align*}
&\int_0^t\int_X\psi(s,z)p(t-s,z,y)\,V(\d{z}) \d{s}\\
&=\int_0^t\int_X\Phi(z)p(t-s,z,y)\,V(\d{z})\d{s}
-\int_X\Phi(w)\left(\int_0^t\int_X\tp(s,w,z)p(t-s,z,y)\,V(\d{z}) \d{s}\right) m(\d{w})\\
&=\Phi(y)-\int_X\tp(t,y,w)\Phi(w)\,m(\d{w})=\psi(t,y).
\end{align*}
\end{proof}

\begin{example}\label{ex:9}
Let $\alpha\in (0,2)$ and $p_D(t,x,y)$, $t>0$, $x,y\in D$, be the Dirichlet heat kernel of $-(-\Delta)^{\alpha/2}$ on $D=(0,\infty)$. 
Let $\tilde{p}_D$ denote the perturbation series for $p_D$ and $V(\d x)=\kappa x^{-\alpha} \d x$, where
\[
\kappa=\kappa_\delta=\frac{\Gamma(\delta+\alpha/2)\Gamma(1-\delta+\alpha/2)}{\Gamma(\delta)\Gamma(1-\delta)}
\]
for some $\delta\in (0,1/2]$.
Then, for $\delta\in(0,1/2)$,
\[
\psi(t,y)=
y^{\frac{\alpha}{2}+\delta-1}
-\int_D \tilde{p}_D(t,y,z)z^{\frac{\alpha}{2}+\delta-1}\,\d z,
\qquad t>0,\; y\in D,
\]
and, for $\delta=1/2$,
\[
\psi(t,y)=-y^{\frac{\alpha-1}{2}}\ln y
+\int_D \tilde{p}_D (t,y,z)z^{\frac{\alpha-1}{2}}\ln z\,\d z,
\qquad t>0,\; y\in D,
\]
is a non-negative, non-trivial solution of \eqref{eq:psi-inv} for $p=p_D$ and $V$.

We now justify this claim.
To make the dependence on the parameter explicit,
we write $\rho_\delta=\tilde{p}_D$, $V_\delta=V$ and $\psi_\delta=\psi$. For $\beta\in(0,1)$, define $h_\beta(x)=x^{\frac{\alpha}{2}-\beta}$, $x\in D$. We then set $f_\delta(x)=h_{1-\delta}(x)$ for $\delta \in(0,1/2)$, and $f_{\frac12}(x)=- h_{\frac12}(x) \ln x$.
Therefore,
\[
\psi_\delta(t,y)=f_\delta(y)-\int_D \rho_\delta(t,y,z)f_\delta(z)\,\d z.
\]
First, we consider $\delta\in (0,1/2)$.
The non-negativity follows by letting $\beta\to (1-\delta)^-$ in
\cite[(17)]{MR4937881} and applying Fatou's lemma.
Thus, $f_\delta$ is super-median for $\rho_\delta$ and satisfies Dynkin's formula for $p_D$ and $\phi=f_\delta V_\delta$ (see Example \ref{Ex:stablehalfline} and Lemma \ref{lem:A1}). Furthermore, $\rho_\delta$ satisfies Duhamel's formula for $p_D$ and $V_\delta$ (see \cite[(13)]{MR4937881}). Since $p_D$ and $\rho_\delta$ are symmetric, Lemma~\ref{lem:psi-construction} applies with $\Phi=f_\delta$, $p=p_D$, $\tp=\rho_\delta$ and $V=V_\delta$, and yields \eqref{eq:psi-inv}.
Finally, by the estimates of $\tp_D$ (\cite[Theorem 1.1]{MR4937881}), $\int_D \tp_D(t,y,z)z^{\frac{\alpha}{2}+\delta-1}\,\d z \approx y^{\alpha/2-\delta}$ as $y \to 0$. Hence, for $\delta \in(0,1/2)$, $\psi(t,y)>0$  for all sufficiently small $y>0$.

We now consider $\delta=1/2$, and we prove the representation
\[
\psi_{\frac12}(t,y)=\lim_{\beta\to 1/2^-}\frac{\psi_\beta(t,y)}{1-2\beta}.
\]
We have
\begin{align*}
\lim_{\beta\to 1/2}
\frac{h_{\beta}(y)-h_{\frac12}(y)}{\beta-1/2}
= - y^{\frac{\alpha-1}{2}}\ln y=f_{\frac12}(y).
\end{align*}
By \cite[(18)]{MR4937881},
$h_\beta(y)=\int_D\rho_\beta(t,y,z)h_{\beta}(z)\d z$, $\beta\in (0,1/2]$. Thus,
\begin{align*}
\lim_{\beta\to 1/2^-}\frac{\psi_\beta(t,y)}{1-2\beta}
	&= \lim_{\beta\to 1/2^-}
\left(
\frac{h_{1-\beta}(y)-h_{\beta}(y)}{1-2\beta}
-\int_D\rho_\beta(t,y,z)
\frac{h_{1-\beta}(z)-h_{\beta}(z)}{1-2\beta}\,\d z\right)\\
	&=f_{\frac12}(y)-\int_D\rho_{\frac12}(t,y,z)f_{\frac12}(z)\,\d z
=\psi_{\frac12}(t,y).
\end{align*}
Here, we use the dominated convergence theorem. The perturbation series $\rho_\beta$ converges to $\rho_{\frac12}$ and is bounded by $\rho_{\frac12}$ with explicit estimates provided by
\cite[Theorem 1.1]{MR4937881}. Moreover, by the
mean value theorem, the difference quotient is bounded by $z^{\frac{\alpha-1}{2}}(z^{\varepsilon}+z^{-\varepsilon})|\ln z|$ if $|\beta-1/2|\leq\varepsilon<1/2$, which leads to  an integrable dominating function if $\varepsilon<\alpha$.
To establish \eqref{eq:psi-inv}, we use the result for $\delta<1/2$ and pass to the limit as follows:
\[
\psi_{\frac12}(t,y)=\lim_{\beta\to 1/2^-}\int_0^t\int_D \frac{\psi_\beta (s,z)}{1-2\beta} p_D(t-s,z,y)\, V_\beta(\d z) \d s.
\]
It suffices to apply the dominated convergence theorem.
To verify its assumptions, we use the bound on the difference quotient as above, the
finiteness of the integral in \cite[Theorem~1.2]{MR4660833}
and, after rewriting the integral expression using Duhamel's formula, the bound on $\rho_{\frac12}$ from \cite[Theorem~1.1]{MR4937881}.
Non-triviality is similar to the case $\delta<1/2$.
\end{example}

We return to Example~\ref{ex:7} and examine it in greater detail. The following proposition and corollary illustrate how the preceding lemmas can be used to show that, for certain parameters, the Bessel kernel admits a perturbation series representation.

\begin{proposition}\label{prop:1}
Let $D=(0, \infty)$, $\delta\in (-\infty,1)$ and
\[
g_\delta (t,x,y) = \frac{\sqrt{xy}}{2t} I_{\frac{1}{2}-\delta} \left( \frac{xy}{2t} \right) \exp\left(-\frac{x^2+y^2}{4t}\right), \qquad t>0,\; x,y \in D.
\]
Then, for all $\delta,\beta \in (-\infty,1)$ such that $\delta+\beta<1$, we have
\[
g_\delta(t,x,y)=g_\beta (t,x,y)+\int_0^t\int_0^\infty g_\delta(s,x,z) \kappa z^{-2} g_\beta (t-s,z,y) \, \d z \d s, \qquad t>0,\; x,y \in D,
\]
where $\kappa = \delta(1-\delta) - \beta(1-\beta)$.
\end{proposition}
\begin{proof}
Let $\nu=1/2-\delta$ and $\mu=1/2-\beta$, and note that $\nu^2-\mu^2=-\kappa$.
By continuity, it suffices to verify the equality for the corresponding Laplace transforms with respect to $t$, which, by \cite[6.653]{MR3307944}, amounts to showing that for $\lambda>0$,
\begin{align}
\begin{aligned}\label{eq:Duh_halfline}
I_\nu(\lambda (x \land y)) & \, K_\nu(\lambda (x \vee y)) =  I_\mu(\lambda (x \land y)) \, K_\mu(\lambda (x \vee y)) \\
&+  \kappa \int_0^\infty \frac{1}{z} \, I_\nu(\lambda (x \land z)) \, K_\nu(\lambda (x \vee z)) \, I_\mu (\lambda (z \land y)) \, K_\mu(\lambda (z \vee y)) \, \d{z}.
\end{aligned}
\end{align}
Without loss of generality, assume that $x\leq y$; the case $x>y$ is analogous. We split the integral into three parts, $J_{(0,x)}$, $J_{(x,y)}$, and $J_{(y,\infty)}$, corresponding to the intervals $(0,x)$, $(x,y)$, and $(y,\infty)$, respectively. We use the identity
\[
\frac{\d}{\d z} \left[ z \left( A_{\nu}'(\lambda z) B_\mu(\lambda z) - A_{\nu}(\lambda z) B_\mu'(\lambda z) \right) \right] = \frac{\nu^2 - \mu^2}{\lambda z} \, A_\nu(\lambda z) B_\mu(\lambda z),
\]
where $A$ and $B$ stand for any of the functions $I$ or $K$. 
This follows by direct differentiation, using the fact that $A_\nu$ and $B_\mu$ satisfy the corresponding modified Bessel equations; see \cite[8.494]{MR3307944}.
Therefore,
\begin{align*}
J_{(0,x)} & = \kappa \int_0^x \frac{1}{z} \, I_\nu(\lambda (x \land z)) \, K_\nu(\lambda (x \vee z)) \, I_\mu(\lambda (z \land y)) \, K_\mu(\lambda (z \vee y)) \, \d{z} \\
 &= \kappa \, K_\nu(\lambda x) K_\mu(\lambda y) \int_0^x \frac{1}{z} I_\nu(\lambda z) I_\mu(\lambda z)  \,\d{z} \\
& = - \lambda x \, K_\nu(\lambda x) K_\mu(\lambda y) 
\left( I_\nu'(\lambda x) I_\mu(\lambda x) - I_\nu(\lambda x) I_\mu'(\lambda x) \right).
\end{align*}
The integrals over $(x,y)$ and $(y,\infty)$ are handled similarly:
\begin{align*}
 J_{(x,y)} &=  I_\nu(\lambda x) K_\mu(\lambda y) \\
		& \quad \times \left( - \lambda y \left[ K_\nu'(\lambda y) I_\mu(\lambda y) - K_\nu(\lambda y) I_\mu'(\lambda y) \right] + \lambda x \left[ K_\nu'(\lambda x) I_\mu(\lambda x) - K_\nu(\lambda x) I_\mu'(\lambda x) \right] \right) 
\end{align*}
and
\begin{align*}		
 J_{(y,\infty)} &= \lambda y \, I_\nu(\lambda x) I_\mu(\lambda y)  \left[ K_\nu'(\lambda  y) K_\mu(\lambda  y) - K_\nu(\lambda  y) K_\mu'(\lambda  y) \right].
\end{align*} 
Using the identity
$I_\nu'(z) K_\nu(z) - I_\nu(z) K_\nu'(z) = \frac{1}{z}$, 
which is the Wronskian identity for
$I_\nu$ and $K_\nu$; see \cite[8.477, 8.486 4. and 13.]{MR3307944}, we simplify most of the resulting terms and obtain \eqref{eq:Duh_halfline}. 
\end{proof}

\begin{corollary}\label{cor:1}
Let $g_D(t,x,y) = \frac{1}{\sqrt{4\pi t}} \left(e^{-\frac{(x-y)^2}{4t}}- e^{-\frac{(x+y)^2}{4t}}\right)$, $t>0$, $x,y\in D$, be the Dirichlet heat kernel of $\Delta$ on $D=(0,\infty)$,
and let $V(\d x) = \delta(1-\delta) x^{-2} \,\d x$ with  $\delta \in (0,1/2]$.  Then, $g_\delta$
is the perturbation series for $g_D$ and $V$.
\end{corollary}

\begin{proof}
By Proposition~\ref{prop:1},
$g_\delta$ satisfies Duhamel's formula for $g_D=g_0$ and $V$.
First, let
$\delta < 1/2$.
From Lemma \ref{lem:omega4} we have $g_\delta = \rho + \omega$. Moreover,
$p_N \to 0$ as $N\to\infty$,
$p_N \leq g_\delta$
and $\omega$ satisfies \eqref{eq:omega-repr}.
We verify that, for fixed $t>0$, $x,y\in D$,
\[
\int_0^t \int_0^\infty g_\delta (s,x,z) g_\delta (t-s,z,y) V(\d z) \d s < \infty
\]
which allows us to apply the dominated convergence theorem and conclude that $\omega=0$.
Let $\varepsilon = (x \land y)/2$.  By the estimates \cite[10.30.4]{DLMF},
\[
g_\delta(u,x,z) \leq c (4 \pi u)^{-1/2}e^{-\frac{(x-z)^2}{4u}},
\qquad u\in (0,t],\; z\in [\varepsilon,\infty).
\]
The analogous estimate holds with $y$ in place of $x$.
Thus,
\[
\int_0^t \int_\varepsilon^\infty g_\delta(s,x,z) g_\delta(t-s,z,y) V(\d z) \d s 
\leq c t V(\varepsilon)  (4\pi t)^{-1/2} e^{-\frac{(x-y)^2}{4t}} <\infty.
\]
Using
\cite[10.30.1 and 10.30.4]{DLMF}, we get $I_\nu(z)\leq c z^\nu e^z$ and
\[
g_\delta(u,x,z) 
	\leq c \left( \frac{xz}{u} \right) ^{1-\delta}  (4 \pi u)^{-1/2}e^{-\frac{(x-z)^2}{4u}}
	\leq c \left( \frac{xz}{u} \right) ^{1-\delta} u^{-1/2}e^{-\frac{x^2}{16 u}},
\qquad u\in (0,t],\; z\in (0,\varepsilon].
\]
Therefore,
\[
\int_0^t \int_0^\varepsilon g_\delta(s,x,z) g_\delta(t-s,z,y) V(\d z) \d s 
	\leq c
\int_0^t
 (xy)^{1-\delta}  s^{\delta-\frac{3}{2}}(t-s)^{\delta-\frac{3}{2}} e^{ -\frac{x^2}{16s}-\frac{y^2}{16(t-s)}}  \int_0^\varepsilon z^{-2\delta}\,\d{z} \d{s}<\infty.
\]
This completes the proof for $\delta<1/2$.
Since the iterated kernels appearing in the perturbation series are non-negative, and $g_\delta$ depends continuously on $\delta \in (0,1/2]$, the identity $g_\delta = \rho$ extends to the critical case $\delta = 1/2$ by the monotone convergence theorem.
\end{proof}

\begin{corollary}\label{cor:2}
In the setting of Corollary~\ref{cor:1}, for
every $\delta\in (0,1/2)$,
we have that $g_{1-\delta}=g_{\delta}+\omega$, where
\begin{align*}
\omega(t,x,y) 
&= 
\sin((1/2-\delta) \pi)
\frac{\sqrt{xy}}{\pi t} K_{\frac12-\delta}\left( \frac{xy}{2t} \right)  \exp\left(-\frac{x^2+y^2}{4t}\right)\\
&= \lim_{N\to \infty} \int_0^t \int_D g_{1-\delta} (s,x,z) p_N(t-s,z,y) \,V(\d z)\d s,\qquad \quad t>0,\; x,y\in D,
\end{align*}
satisfies \eqref{eq:omega-inv} with $p=g_D$ and $V$.
\end{corollary}
\begin{proof}
By Proposition~\ref{prop:1},
$g_{1-\delta}$ satisfies Duhamel's formula for $g_D$ and $V$.
Thus,  Lemma~\ref{lem:omega4}
guarantees that $g_{1-\delta}=\rho+\omega$,
where $\rho$ is the perturbation series, and $\omega$ satisfies \eqref{eq:omega-inv} and \eqref{eq:omega-repr} for the corresponding objects.
According to Corollary~\ref{cor:1},  $\rho=g_\delta$.
We get $\omega=g_{1-\delta}-g_\delta$ and it suffices to express
$I_{\delta-\frac12} -I_{\frac12-\delta}
$
using $K_{\frac12-\delta}$, see \cite[8.485]{MR3307944}.
\end{proof}

We provide a few observations.

\begin{remark}
Recall that the perturbation series satisfies Duhamel's formula (see Lemma~\ref{lem:bp-Duhamel}). However, in general, Duhamel's formula does not imply a perturbation series representation (see Lemma~\ref{lem:D+w=D} and Examples~\ref{ex:6}--\ref{ex:9}), even when restricted to transition densities
(see Proposition~\ref{prop:1} and Corollary~\ref{cor:2}).
\end{remark}

\begin{remark}\label{rem:4}
Let us observe that in Example~\ref{ex:6}, $\psi(t,y)= \lim_{|x|\to 0} t^\delta |y|^{-\delta} g(t,x,y)$, while in Example~\ref{ex:7}, $\psi(t,y) = \lim_{x\to 0^+} t^\delta y^{-\delta} g_D(t,x,y)/x$. However, we regard these representations as specific to the examples above rather than as a general rule. 
We believe that the approach described in Remark~\ref{rem:psi-inv}, based on dividing $\tp$ by a suitably chosen function, may provide a more general way of obtaining functions $\psi$. In Example~\ref{ex:7}, this can be verified by direct computation:
$\psi(t,y)= c_\delta \lim_{x\to 0^+} g_\delta(t,x,y)/x^{1-\delta}$.
A similar representation is expected to hold in Example~\ref{ex:6}. This viewpoint is further supported by Example~\ref{ex:8}, where $\psi$ is defined precisely by such a procedure.
\end{remark}

\section{Positive perturbations}\label{sec:D-Pos}

\subsection{Via Duhamel's formula}\label{subsec:D-Pos}

\begin{theorem}\label{prop:tp-init}
Assume that $\Phi$ satisfies Dynkin's formula for $p$ and $\phi$, and that  
$\tp$ satisfies Duhamel's formula for $p$ and $V$.
Then,
for all $t>0$, $x\in X$,
\[
\int_X \tp(t,x,y)\Phi(y)\, m(\d y) + \int_0^t \int_X \tp(s,x,y)\, \phi( \d y) \d s = 
\Phi(x)+ \int_0^t \int_X \tp(s,x,y)  \Phi(y)\, V(\d y) \d s \,.
\]
\end{theorem}
\begin{proof}
By Duhamel’s formula \eqref{eq:tp-Duh},
\begin{align*}
\int_X \tp(t,x,y)\Phi(y)\, m(\d y) 
&= \int_X p(t,x,y)\Phi(y) \, m(\d y) \\
&\quad + \int_0^t \int_X \tp(u,x,z) \int_X p(t-u,z,y) \Phi(y) \, m(\d y)\, V(\d z) \d u 
\end{align*}
and
\begin{align*}
\int_0^t \int_X \tp(s,x,y)\, \phi(\d y) \d s
&= \int_0^t \int_X p(s,x,y) \, \phi(\d y) \d s\\
& + \int_0^t \int_X \left(  \int_0^s \int_X \tp(u,x,z) p(s-u,z,y)\, V(\d z) \d u \right) \, \phi(\d y) \d s\\
&= \int_0^t \int_X p(s,x,y)\phi(\d y) \d s\\
& + \int_0^t \int_X  \tp(u,x,z) \left(\int_0^{t-u} \int_X p(r,z,y)\, \phi(\d y) \d r \right) V(\d z) \d u.
\end{align*}
It suffices to add both sides and apply Dynkin's formula \eqref{eq:init} twice.
\end{proof}

\begin{remark}
If we replace ''$=$'' by ''$\leq$'' in Dynkin's formula  \eqref{eq:init}, then the statement of  Theorem~\ref{prop:tp-init} remains valid with ''$\leq$'' in place of ''$=$''. The proof proceeds along similar lines.
\end{remark}

Corollary~\ref{cor:eq_ineq} is the key ingredient for future applications and will also be used in the proof of Theorem~\ref{thm:super_median}. To the best of our knowledge, the first result of this kind was established in \cite[Theorem~3.1]{MR3933622} for the fractional Laplacian with the Hardy potential.

\begin{corollary}\label{cor:eq_ineq}
Assume the hypotheses of Theorem~\ref{prop:tp-init}. Fix $t>0$ and $x\in X$. If
\begin{align}\label{eq:eq_ineq}
    \int_0^t \int_X \tp(s,x,y)\Phi(y)\,V(\d y)\,\d s<\infty,
\end{align}
then
\begin{align}\label{eq:3}
\int_X \tp(t,x,y)\Phi(y)\,m(\d y)
+\int_0^t\int_X \tp(s,x,y)\big[\phi(\d y)-\Phi(y)V(\d y)\big]\,\d s
=\Phi(x).
\end{align}
In particular,
\begin{align*}
\phi(\d y)\ge \Phi(y)V(\d y)
&\quad\Longrightarrow\quad
\int_X \tp(t,x,y)\Phi(y)\,m(\d y)\le \Phi(x),\\
\phi(\d y)= \Phi(y)V(\d y)
&\quad\Longrightarrow\quad
\int_X \tp(t,x,y)\Phi(y)\,m(\d y)= \Phi(x).
\end{align*}
\end{corollary}

The significance of the condition \eqref{eq:eq_ineq} is discussed in Remark~\ref{rem:3}. We shall view \eqref{eq:3} as an extended version of Dynkin's formula. Equation~\eqref{eq:i-2} already indicates that the sign of
$
\phi(\d y)-\Phi(y)V(\d y)
$
may depend on the parameter~$\beta$. In full generality, the situation is even more delicate, since this expression need not have a fixed sign.

\subsection{Via perturbation series}\label{sec:PS-P}

In this section we consider the perturbation series $\rho$ as defined by \eqref{eq:pertser}
and \eqref{eq:pn}.
We first prove a version of Theorem~\ref{prop:tp-init} for $p_n$, cf. 
\cite[Lemma~3.6]{MR3933622}.

\begin{lemma}\label{lem:p_n}
Assume that $\Phi$ satisfies Dynkin's formula for $p$ and $\phi$. Then,
for all $t>0$, $x\in X$ and $n\in\mathbb{N}\cup \{0\}$, 
\[
\int_X p_{n+1}(t,x,y) \Phi(y)\, m(\d y)
+
\int_0^t\int_X p_{n+1}(s,x,y)\, \phi(\d y) \d s =
\int_0^t \int_X p_n(s,x,y) \Phi(y) \, V(\d y) \d s .
\]
\end{lemma}
\begin{proof}
By the definition \eqref{eq:pn} of $p_{n+1}$,
\begin{align*}
\int_X p_{n+1}(t,x,y) \Phi(y) \, m(\d y)
=\int_0^t\int_X p_n(u,x,z) \int_X  p(t-u,z,y) \Phi(y) \, m(\d y) V(\d z)\d u
\end{align*}
and
\begin{align*}
\int_0^t &\int_X p_{n+1} (s,x,y)\, \phi(\d y) \d s \\
&= \int_0^t \left( \int_X \int_0^s \int_X p_n(u,x,z) p(s-u,z,y)\, V(\d z) \d u \right) \, \phi(\d y) \d s \\
&= \int_0^t \int_X p_n (u,x,z)\left( \int_0^{t-u} \int_X p(r,z,y)\, \phi(\d y) \d r \right) V(\d z) \d u.
\end{align*}
It suffices to add both sides and apply Dynkin's formula \eqref{eq:init}.
\end{proof}

We provide a useful general condition on $V$ which makes Corollary~\ref{cor:eq_ineq} widely applicable.

\begin{corollary}\label{cor:finite_int}
Assume that $\Phi$ satisfies Dynkin's formula  for $p$ and $\phi$.
Suppose that $\Phi(x)V(\d x) \leq (1-\varepsilon)\phi(\d x)$ for some $\varepsilon\in(0,1]$. Then, for all $x\in X$ and $n\in\mathbb{N}\cup \{0\}$,
\[
\int_0^\infty \int_X p_n(s,x,y) \, \phi(\d y) \d s \leq (1-\varepsilon)^n \Phi(x)
\]
and
\[
\int_0^\infty \int_X \rho(s,x,y)\,\phi(\d y) \d s \leq \frac1\varepsilon \Phi(x) .
\]
\end{corollary}
\begin{proof}
For $n=0$ the inequality is true by Dynkin's formula \eqref{eq:init}.
Assume that it holds for $n\in\mathbb{N}\cup \{0\}$. Then, by Lemma~\ref{lem:p_n},
\begin{align*}
\int_0^\infty \int_X p_{n+1}(s,x,y)\, \phi(\d y) \d s
&\leq \int_0^\infty \int_X p_n(s,x,y)\Phi(y)\, V(\d y) \d s \\
&\leq (1-\varepsilon) \int_0^\infty \int_X p_n(s,x,y) \, \phi(\d y) \d s
\leq (1-\varepsilon)^{n+1} \Phi(x).
\end{align*}
Thus, the first inequality follows by induction. The second inequality follows by summing the first inequality over all $n$.
\end{proof}

We show that if $\Phi$ satisfies Dynkin's formula for the measure $\phi$, then it is  super-median for the perturbation series associated with the potential $V$ such that $\Phi V \leq \phi$, see Remark~\ref{rem:2}.

\begin{theorem}\label{thm:super_median}
Assume that $\Phi$ satisfies Dynkin's formula  for $p$ and $\phi$, and $V$ is such that
\[
\Phi(x) V(\d x)\leq \phi(\d x).
\]
If $\rho$ is the perturbation series for $p$ and $V$, then
$\Phi$ is super-median for $\rho$, i.e.,
\[
\int_X \rho(t,x,y)\Phi(y)\, m(\d y) \leq \Phi(x), 
\quad
t>0,\; x\in X.
\]
\end{theorem}
\begin{proof}
If $\Phi(x)=\infty$, then the inequality holds trivially. Suppose that $\Phi(x)<\infty$.
We denote by $\rho_{V_\varepsilon}$ the  perturbation series for $p$ and $V_\varepsilon=(1-\varepsilon)V$. Then, by Corollary~\ref{cor:finite_int},
\[
\int_0^\infty \int_X \rho_{V_\varepsilon}(s,x,y)\, \phi(\d y) \d s <\infty\,.
\]
Thus, by Corollary~\ref{cor:eq_ineq},
\[
\int_X \rho_{V_\varepsilon}(t,x,y)\Phi(y)\, m(\d y)  \leq \Phi(x)\,.
\]
Finally, we let $\varepsilon\downarrow 0^+$ and use the monotone convergence theorem. 
\end{proof}

We collect a few observations.

\begin{remark}\label{rem:2}
Given a function $\Phi$ and a measure $\phi$, we let $\xi(x)=0$ if $\Phi(x)=0$ or $\Phi(x)=\infty$; otherwise,  set
$\xi(x)=1/\Phi(x)$. Define
\[
V(\d x)=\xi(x)\phi(\d x).
\]
By construction,  $\Phi(x)\xi(x)\leq 1$.
Thus, $\Phi(x) V(\d x)\leq \phi(\d x)$,
as required in Theorem~\ref{thm:super_median}.
\end{remark}

\begin{remark}
If we replace ''$=$'' by ''$\leq$'' in Dynkin's formula  \eqref{eq:init}, then the statement of  Theorem~\ref{thm:super_median} remains valid. The proof proceeds in a similar manner.
\end{remark}

\begin{remark}\label{rem:3}
For $\beta\in (0,1)$, we let $\Phi_\beta(x)=x^{1-\beta}$ and $\phi_\beta(\d x)=\beta(1-\beta)x^{-1-\beta}\d x$. Let $\delta \in (0,1/2]$.
By Proposition~\ref{prop:1}, both $g_\delta$ and $g_{1-\delta}$ satisfy Duhamel's formula for $g_D=g_0$ and 
\[
V(\d x)=\frac{\phi_\delta(\d x)}{\Phi_{\delta}(x)}=\frac{\phi_{1-\delta}(\d x)}{\Phi_{1-\delta}(x)}.
\]
At the same time, by Example~\ref{Ex:gausshalfline} and
Lemma~\ref{lem:A1},  $\Phi_\delta$ satisfies Dynkin's formula for $g_D$ and $\phi_\delta$, while $\Phi_{1-\delta}$ satisfies Dynkin's formula for $g_D$ and $\phi_{1-\delta}$.
In light of the discussion in the Introduction, we note that
\[
\int_D g_\delta(t,x,y) y^{1-\delta}\, \d{y}  = x^{1-\delta},
\qquad t>0,\; x\in D,
\]
which follows either by direct computation or, for $\delta\in(0,1/2)$, by Corollary~\ref{cor:eq_ineq}; the case $\delta=1/2$ is then obtained by taking the limit.
This means that $\Phi_\delta$ is invariant for $g_\delta$.
Recall that $g_\delta$ is the perturbation series for $g_D$ and $V$ (see Corollary~\ref{cor:1}).
Interestingly,
a direct computation also gives
\[
\int_D g_{1-\delta} (t,x,y) y^{\delta}\, \d{y}  = x^{\delta},
\qquad t>0,\; x\in D,
\]
that is, the function $\Phi_{1-\delta}$ is invariant for $g_{1-\delta}$. However, for $\delta \in (0,1/2)$, the latter is not a perturbation series for $g_D$ and $V$ (see Corollary~\ref{cor:2}); as mentioned above, it does satisfy Duhamel's formula for $g_D$ and $V$.
Now let $\delta\in (0,1/2)$. Since $g_{1-\delta}>g_\delta$ by Corollary~\ref{cor:2}, we have
\[
\int_D g_\delta (t,x,y) y^{\delta}\, \d{y}  < x^{\delta},
\qquad t>0,\; x\in D,
\]
and
\[
\int_D g_{1-\delta} (t,x,y) y^{1-\delta}\, \d{y}  > x^{1-\delta},
\qquad t>0,\; x\in D.
\]
Thus, the function
$\Phi_{1-\delta}$ is super-median for $g_\delta$, but is not invariant.  In particular, even after strengthening the assumption in Theorem~\ref{thm:super_median} to $\Phi(x) V(\d x)= \phi(\d x)$, we cannot, in general, 
expect equality in the conclusion, namely, that
$\Phi$ is invariant for the perturbation series $\rho$.
The last inequality shows that the condition \eqref{eq:eq_ineq} in Corollary~\ref{cor:eq_ineq} cannot be simply dropped.
Indeed, the hypotheses of Theorem~\ref{prop:tp-init}
are satisfied with $\Phi=\Phi_\delta$,
$p=g_D$, $\phi(\d{x})=\phi_\delta(\d{x})$,
and $\tilde{p}=g_{1-\delta}$, but \eqref{eq:3} still fails.
\end{remark}

Here is how we recover \cite[Theorem~1]{MR3460023} using Theorem~\ref{thm:super_median}.
\begin{example}\label{ex:2} 
Let $f : \RR\ \to [0,\infty)$ be non-decreasing, and assume that $f=0$ on $(-\infty,0]$.
Let $\mu$ be a $\sigma$-finite measure on $(X,\mathcal{M})$.
Define $h$ by \eqref{eq:hbyf}.
Furthermore, let 
$q(x)=0$ if $h(x)=0$ or $h(x)=\infty$, and otherwise set
\[
q(x)=\frac1{h(x)} \int_0^\infty \int_X f'(t) p(t,x,y)\, \mu(\d y) \d t.
\]
Denote by $\rho_q$ the perturbation series for $p$ and $q(x)m(\d x)$.
If $p$ is a transition density, then
\[
\int_X \rho_q (t,x,y)h(y)\, m(\d y)  \leq h(x).
\]
Indeed, by Lemmas~\ref{ex:1} and~\ref{lem:A1}, $h$ satisfies Dynkin’s formula for $p$ and the measure $\theta$ defined in Lemma~\ref{ex:1}. Since $f'(t)\d t$ is dominated by the Stieltjes measure $\lambda(\d t)$ associated with $f$, the definition of $q$ gives $h(x)q(x)m(\d x)\leq\theta(\d x)$. Theorem~\ref{thm:super_median} applied with $\Phi=h$, $\phi=\theta$, and $V(\d x)=q(x)m(\d x)$, yields the asserted inequality.
\end{example}

\section{Negative perturbations via Duhamel's formula}\label{sec:D-Neg}

In this section we develop a counterpart of the results of Section~\ref{sec:D-Pos} that is suited to the study of non-positive (negative) perturbations.

\begin{definition}
\rm
We say that 
an $\mathcal{M}$-measurable
function $\Psi: X \to [0,\infty]$ satisfies (sub-median)  \textit{Dynkin's formula} for $p$ and
$\psi$ provided
$\psi$ is a measure on $(X,\mathcal{M})$, and
\begin{align}\label{eq:init-n}
\int_X p(t,x,y)\Psi(y)\, m(\d y) 
=\Psi(x)+ \int_0^t \int_X p(s,x,y)\, \psi(\d y) \d s,
\qquad  t>0, \; x\in X.
\end{align}
\end{definition}

\begin{definition}
\rm
We say that
$\tp$  satisfies \textit{Duhamel's formula}
for $p$ and $-U$ provided
$U$ is a measure on $(X,\mathcal M)$, and
\begin{align}\label{eq:Duh-n}
\tp(t,x,y)+\int_0^t\int_X \tp(s,x,z)p(t-s,z,y)\, U(\d z)\d s
=p(t,x,y),
\qquad t>0,\, x,y\in X.
\end{align}
Whenever this holds, we say that
$\tp$ is a Schr\"odinger perturbation of $p$ by $-U$.
\end{definition}

We discuss the sub-median version of Dynkin's formula~\eqref{eq:init-n} and ways of finding $\Psi$ and $\psi$ that satisfy it. A general construction is given in Lemma~\ref{lem:BH}: subtracting a zero-resolvent from a dominating invariant function reverses the sign of the source term in Dynkin's formula and yields a non-negative function. Example~\ref{ex:4.1} illustrates this result. Examples~\ref{ex:4.3} and~\ref{ex:4} indicate a different approach, see also \cite{MR4140086}.

\begin{lemma}\label{lem:BH}
Assume that
$\Omega$ is invariant for $p$, and
$\Theta$ is the zero-resolvent for $p$ and $\theta$.
If $p$ is a transition density, $\Omega(x)<\infty$ for every $x\in X$, and $\Theta\leq \Omega$, then
\[
\Psi(x)=\Omega(x)-\Theta(x),\qquad x \in X,
\]
is non-negative and satisfies (sub-median) Dynkin's formula for $p$ and $\psi=\theta$.
\end{lemma}
\begin{proof}
By Lemma \ref{lem:A1}, $\Theta$ satisfies (super-median) Dynkin's formula for $p$ and $\theta$. Since $0\leq\Theta\leq \Omega<\infty$ and $\Omega$ is invariant for $p$, all terms below are finite. Thus, for every $t>0$ and $x\in X$,
\begin{align*}
\int_X p(t,x,y)\Psi(y)\,m(\d y)
&=\Omega(x)-\int_X p(t,x,y)\Theta(y)\,m(\d y)\\
&=\Omega(x)-\Theta(x)+\int_0^t\int_X p(s,x,y)\,\theta(\d y)\d{s}\\
&=\Psi(x)+\int_0^t\int_X p(s,x,y)\,\theta(\d y)\d{s}.
\end{align*}
\end{proof}

\begin{example}\label{ex:4.1}
Let $d\in\mathbb{N}$ satisfy $d\geq 3$ and $g(t,x,y)=g(t,y-x)= (4\pi t)^{-d/2}e^{-\frac{|y-x|^2}{4t}}$, $t>0$,~$x,y\in\Rd$.
For $\beta>0$, put
\[
a_\beta=\frac{d-2}{d-2+\beta},\qquad
\theta(\d x)=a_\beta\beta(d-2+\beta)|x|^{\beta-2}\mathbf{1}_{B(0,1)}(x)\,\d x.
\]
The function
\[
\Psi(x) =\begin{cases}
a_\beta|x|^\beta,& |x|\leq1,\\
1-(1-a_\beta)|x|^{2-d},& |x|>1,
\end{cases}
\]
satisfies (sub-median) Dynkin's formula for $g$ and $\theta$.

Indeed, the function $\Omega\equiv1$ is invariant for $g$.  Let
\[
\Theta(x)=\int_0^\infty\int_{\mathbb R^d}g(t,x,y)\,\theta(\d y) \d{t}.
\]
Using the mean value property of the Newtonian kernel, we obtain
\[
\Theta(x)=\beta\int_0^1 r^{d+\beta-3}(|x|\vee r)^{2-d} \,\d r 
=\begin{cases}
1-a_\beta|x|^\beta,& |x|\leq1,\\
(1-a_\beta)|x|^{2-d},& |x|>1.
\end{cases}
\]
In particular, $0\leq\Theta\leq\Omega$. The claim follows from Lemma~\ref{lem:BH}.
\end{example}

Lemma~\ref{lem:BH} applies whenever the zero-resolvent of $\theta$ is dominated by a finite invariant function. This domination need not hold in the examples below: the power functions in Examples \ref{ex:4.3} and \ref{ex:4} grow faster than the natural invariant functions of the respective kernels.
Establishing the sub-median version of Dynkin’s formula for these functions therefore requires a different argument, based on compensated Green kernels in Example~\ref{ex:4.3} and compensated Gaussian kernels in Example~\ref{ex:4}. Once this formula is established, Theorem~\ref{prop:tp-init-n} yields a corresponding perturbed Dynkin-type identity, subject to an explicit finiteness condition.

\begin{theorem}\label{prop:tp-init-n}
Assume that $\Psi$ satisfies (sub-median) Dynkin's formula for $p$ and $\psi$, and that  
$\tp$ satisfies Duhamel's formula for $p$ and $-U$.
Let $t>0$, $x\in X$ be fixed. If
\begin{align}\label{ineq:finite-n}
\int_0^t \int_X \int_0^r\int_X \tp(s,x,z)  p(r-s,z,y) \, U(\d z) \d s \,\psi(\d y)\d r <\infty,
\end{align}
then
$$
\int_X \tp(t,x,y)\Psi(y)\, m(\d y) + \int_0^t \int_X \tp(s,x,y)  \Psi(y)\, U(\d y) \d s  = 
\Psi(x)+ \int_0^t \int_X \tp(s,x,y)\, \psi( \d y) \d s.
$$
\end{theorem}
\begin{proof}
We integrate \eqref{eq:Duh-n} with respect to $\Psi(y)m(\d y)$,
\begin{align*}
\int_X \tp(t,x,y)\Psi(y)m(\d y) 
+\int_0^t\int_X \tp(s,x,z)\int_X p(t-s,z,y) \Psi(y)m(\d y)\, U(\d z)\d s& \\
=\int_X p(t,x,y) \Psi(y)m(\d y)&.
\end{align*}
Now, we use \eqref{eq:init-n} twice:
for the inner integral on the left-hand side and for the integral on the right-hand side above,
\begin{align*}
&\int_X \tp(t,x,y)\Psi(y)m(\d y)
+\int_0^t\int_X \tp(s,x,z)\Psi(z) \, U(\d z)\d{s}\\
&\qquad +\int_0^t\int_X \tp(s,x,z)\int_0^{t-s}\int_X p(u,z,y) \, \psi(\d y)\d u\, U(\d z)\d s
=\Psi(x)+\int_0^t \int_X p(r,x,y)\, \psi(\d y)\d r.
\end{align*}
Since 
\begin{align*}
&\int_0^t\int_X \tp(s,x,z)\int_0^{t-s}\int_X p(u,z,y) \, \psi(\d y)\d u\, U(\d z)\d s\\
&\qquad =
\int_0^t\int_X \int_s^t\int_X \tp(s,x,z) p(r-s,z,y) \, \psi(\d y)\d r\, U(\d z)\d s\\
&\qquad =
\int_0^t\int_X \int_0^r\int_X \tp(s,x,z) p(r-s,z,y) \, U(\d z)\d s \, \psi(\d y)\d r<\infty
\end{align*}
we subtract it from both sides to get
\begin{align*}
&\int_X \tp(t,x,y)\Psi(y)m(\d y)
+\int_0^t\int_X \tp(s,x,z)\Psi(z) \, U(\d z)\d s\\
&=\Psi(x)+\int_0^t \int_X p(r,x,y)\, \psi(\d y)\d r
-\int_0^t\int_X \int_0^r\int_X \tp(s,x,z) p(r-s,z,y) \, U(\d z)\d s \, \psi(\d y)\d r\\
&=\Psi(x)+\int_0^t \int_X \tp(r,x,y)\, \psi(\d y)\d r.
\end{align*}
In the last equality we used \eqref{eq:Duh-n} for $t=r$  integrated with respect to $\psi(\d y)\d r$, together with the finiteness assumption.
\end{proof}

\begin{corollary}\label{rem:tp-init-n}
Assume the hypotheses of Theorem~\ref{prop:tp-init-n}.
Let $t>0$, $x\in X$ be fixed. If
\begin{align}\label{ineq:p-finite-n}
\int_0^t\int_X p(s,x,y)\, \psi(\d y)\d s<\infty
\end{align}
then \eqref{ineq:finite-n} is satisfied, and we have
\[
\int_X \tp(t,x,y)\Psi(y)\, m(\d y) + \int_0^t \int_X \tp(s,x,y) \big[ \Psi(y)\, U(\d y) - \psi( \d y) \big] \d s  = 
\Psi(x).
\]
In particular,
\[
\psi(\d y)= \Psi(y)U(\d y)
\quad\Longrightarrow\quad
\int_X \tp(t,x,y)\Psi(y)\,m(\d y)= \Psi(x).
\]
\end{corollary}
\begin{proof}
The left-hand side of \eqref{ineq:p-finite-n} dominates that of \eqref{ineq:finite-n}, as follows from Duhamel's formula \eqref{eq:Duh-n}.
The latter also gives $\tp\leq p$,
and due to \eqref{ineq:p-finite-n} 
we can subtract the finite integral from both sides in the statement of Theorem~\ref{prop:tp-init-n}.
\end{proof}

\begin{example}\label{ex:4.3}
Let $g_D(t,x,y) = \frac{1}{\sqrt{4\pi t}} \left(e^{-\frac{(x-y)^2}{4t}}- e^{-\frac{(x+y)^2}{4t}}\right)$, $t>0$, $x,y\in D$, be the Dirichlet heat kernel of $\Delta$ on $D=(0,\infty)$. Then, for all $\beta>0$, $t>0$  and $x\in D$,
\[
\int_D g_D(t,x,y) y^{1+\beta} \d{y}=x^{1+\beta}+ \int_0^t \int_D g_D(s,x,y) \beta(1+\beta) y^{\beta-1} \, \d{y} \d{s}.
\]
Equivalently, in the Doob conditioned setting with $m(\d x)=x^2 \d x$,
\[
\int_D \frac{g_D(t,x,y)}{x \, y} y^\beta \, m(\d{y})=
 x^{\beta} + \int_0^t \int_D \frac{g_D(s,x,y)}{x \, y} \beta(1+\beta) y^{\beta-2} \, m(\d{y}) \d{s}.
\]
Furthermore, using the notation of Proposition~\ref{prop:1}, for all $\delta>0$, $t>0$ and $x\in D$,
\[
\int_D g_{-\delta} (t,x,y) y^{1+\delta}\,\d{y} = x^{1+\delta}.
\]

We give a justification. Consider
$g_D^0(t,y)=\lim_{x\to 0^+} g_D(t,x,y)/x$
and
\[
F(x,y) = \int_0^\infty \big(x g_D^0(s,y)-  g_D(s,x,y)\big)\, \d{s} =
x -  (x\land y) \geq 0.
\]
Then, for every $\beta>0$,
\[
\int_D F(x,y)y^{\beta-1} \, \d{y} = \frac{x^{1+\beta}}{\beta(\beta+1)}
\]
and, since $\int_D g_D(t,x,z) z \, \d{z} = x$, we also have
\begin{align*}
\int_D g_D(t,x,z) F(z,y) \,\d{z}
&= \int_D g_D(t,x,z)\left(z - \int_0^\infty g_D(s,z,y)\,\d{s} \right) \d{z}\\
&= x - \int_0^\infty g_D(t+s,x,y) \,\d{s}\\
&= F(x,y) + \int_0^t g_D(s,x,y)\, \d{s}.
\end{align*}
It remains to multiply both sides by $y^{\beta-1}$
and integrate over $D$. 
We have therefore justified that $x^{1+\delta}$ satisfies (sub-median) Dynkin's formula for $g_D$ and $\psi(\d x)=\delta(1+\delta)x^{\delta-1}\d x$. Put $U(\d x)=\delta(1+\delta)x^{-2}\d x$, so that $\psi(\d x)=\Psi(x)U(\d x)$. By Proposition~\ref{prop:1}, $g_{-\delta}$ satisfies Duhamel's formula for $g_D$ and $-U$.
The condition \eqref{ineq:p-finite-n} follows from Dynkin's formula and   $\int_D g_D(t,x,y)y^{1+\delta} \d y<\infty$.
Corollary~\ref{rem:tp-init-n} then gives the invariance.
\end{example}

We will justify \eqref{eq:init-n} for power functions and the Gauss–Weierstrass kernel.
First, we extend the result in \eqref{eq:gauss-1} to the range $\gamma > d/2 - 1$.
For $n=1$, the statement of the following lemma can be found in \cite[Proof of Lemma~2.2]{MR4140086}.

\begin{lemma}\label{lem:gauss-1}
Let $d\geq 3$. For all $n\in\mathbb{N}$, $\gamma\in (d/2-2+n,d/2-1+n)$ and $x\in\Rd$,
\begin{align*}
\int_0^\infty (-1)^n\left[g(s,x)- g(s,0) \sum_{k=0}^{n-1}\frac1{k!}\left(-\frac{|x|^2}{4s}\right)^k   \right] s^\gamma \, \d s
= \frac1{4^{\gamma+1}\pi^{d/2}} \frac{\Gamma(d/2-1+n-\gamma)}{\prod_{j=1}^n (\gamma-d/2+2-j)} |x|^{2\gamma-d+2}.
\end{align*}
\end{lemma}
\begin{proof}
At $x=0$ both sides vanish, so we assume that $x\ne0$. The convergence of the integral follows from the following calculation. We substitute $t=|x|^2/(4s)$ in the integral on the left-hand side to get
\begin{align*}
(4\pi)^{-d/2} &\left(\frac{|x|^2}{4}\right)^{\gamma-d/2+1}
\int_0^\infty (-1)^n\left[e^{-t}-\sum_{k=0}^{n-1}\frac{(-t)^k}{k!}\right]  t^{d/2-2-\gamma} \,\d t\\
&=
\frac{|x|^{2\gamma-d+2}}{4^{\gamma+1}\pi^{d/2}}
\int_0^\infty\left[ \int_0^{u_n}\int_0^{u_{n-1}}\ldots\int_0^{u_1} e^{-u_0}\, \d u_0\ldots \d u_{n-1} \right] u_n^{d/2-2-\gamma} \,\d u_n \\
&=
\frac{|x|^{2\gamma-d+2}}{4^{\gamma+1}\pi^{d/2}}
\int_0^\infty\left[ \int_{u_0}^{\infty}\int_{u_1}^{\infty}\ldots\int_{u_{n-1}}^{\infty} e^{-u_0} u_n^{d/2-2-\gamma}\, \d u_n\ldots \d u_{0} \right]\\
&=
\frac{|x|^{2\gamma-d+2}}{4^{\gamma+1}\pi^{d/2}}
\left(\prod_{j=1}^n\frac1{\gamma-d/2+2-j}\right)
\int_0^\infty  e^{-u_0}  u_0^{d/2-2+n-\gamma} \,\d u_0.
\end{align*}
\end{proof}

\begin{example}\label{ex:4}
Let $d\geq 3$ and $g(t,x,y)=g(t,y-x)= (4\pi t)^{-d/2}e^{-\frac{|y-x|^2}{4t}}$ where $t>0$, $x,y\in\Rd$. 
Then, for every $\beta > 0$,
\begin{align}\label{eq:2}
\int_{\Rd} g(t,x,y) |y|^\beta \,\d y 
=|x|^\beta + \beta(d-2+\beta) \int_0^t \int_{\Rd} g(s,x,y)\, |y|^{\beta-2} \, \d y \d s,
\qquad t>0,\, x\in\Rd.
\end{align}

Here is a justification: 
for $n\in\mathbb{N}$ and $2(n-1)<\beta<2n$,
we let
\[
f(t)=c t^{\frac{d-2+\beta}{2}}, \qquad c= \frac{2^{d+\beta}\pi^{d/2}}{\Gamma(n-\beta/2)}  {\prod_{j=1}^n (\beta/2+1-j)}.
\]
By Lemma~\ref{lem:gauss-1} with $\gamma=d/2-1+\beta/2$,
\[
\int_0^\infty (-1)^n\left[g(s,x)- g(s,0) \sum_{k=0}^{n-1}\frac1{k!}\left(-\frac{|x|^2}{4s}\right)^k   \right] f(s) \, \d s= |x|^{\beta}.
\]
Since $2(n-2)<\beta-2<2(n-1)$,
by Lemma~\ref{lem:gauss-1} with $\gamma=d/2-1+(\beta-2)/2$
and by \eqref{eq:gauss-1},
\[
\int_0^\infty (-1)^{n-1}\left[g(s,x)- g(s,0) \sum_{k=0}^{n-2}\frac1{k!}\left(-\frac{|x|^2}{4s}\right)^k   \right] f'(s) \, \d s  = \beta (d-2+\beta)  |x|^{\beta-2}. 
\]
We use the usual convention that the sum is zero when the upper summation index is smaller than the lower one. We have also used the relation
\[
c \left(\frac{d-2+\beta}{2} \right) \frac{\Gamma(n-\beta/2)}{2^{d+\beta-2}\pi^{d/2}}  \frac1{\prod_{j=1}^{n-1} (\beta/2-j)} =\beta(d-2+\beta).
\]
We proceed by induction.
Assume that the equality holds for $\beta \leq 2(n-1)$.
We first consider $\beta \in (2(n-1),2n)$.
We set $\mathcal{E}=\int_{\Rd} g(t,x,y) |y|^\beta \,\d y 
-|x|^\beta$, and for $x\neq 0$ we have
\begin{align*}
	\mathcal{E}&= 
\int_{\Rd} g(t,x,y) \int_0^\infty (-1)^n\left[g(s,y)- g(s,0) \sum_{k=0}^{n-1}\frac1{k!}\left(-\frac{|y|^2}{4s}\right)^k   \right] f(s) \, \d s \, \d y\\
&\quad -\int_0^\infty (-1)^n\left[g(s,x)- g(s,0) \sum_{k=0}^{n-1}\frac1{k!}\left(-\frac{|x|^2}{4s}\right)^k   \right] f(s) \, \d s\\
	&=\int_0^\infty (-1)^n\bigg[g(s+t,x)-g(s,x)\\
&\quad - g(s,0) \sum_{k=1}^{n-1}\frac1{k!}\left(-\frac1{4s}\right)^k 2k(d-2+2k)\int_0^t \int_{\Rd} g(u,x,z)|z|^{2(k-1)}\,\d z \d u   \bigg] f(s) \, \d s.
\end{align*}
The latter integral is absolutely convergent, as it is the difference of two absolutely convergent integrals.
We write
$f(s)=\int_0^s f'(r)\,\d r$, and we change the order of integration, to get
\begin{align*}
\mathcal{E}&=
 \int_0^\infty \int_r^\infty (-1)^n\bigg[ g(s+t,x)-g(s,x) \\
 &\quad - g(s,0) \sum_{k=1}^{n-1}\frac1{(k-1)!}\left(-\frac1{4s}\right)^k 2(d-2+2k)\int_0^t \int_{\Rd} g(u,x,z)|z|^{2(k-1)}\,\d z \d u    \bigg]  \d s \, f'(r) \, \d r.
\end{align*}
Since
\[
\int_r^\infty (g(s+t,x)-g(s,x)) \, \d s= -\int_0^t g(r+u,x) \, \d u = - \int_0^t \int_{\Rd} g(u,x,z) g(r,z) \, \d z \d u
\]
and
\[
2(d-2+2k)\int_r^\infty g(s,0) s^{-k} \, \d s =  4 \,g(r,0) r^{-(k-1)}
\]
we have
\begin{align*}
	\mathcal{E}&=
\int_0^t \int_{\Rd} g(u,x,z)
 \int_0^\infty (-1)^{n-1}\bigg[  g(r,z) 
- g(r,0) \sum_{k=1}^{n-1}\frac{(-1)^{k-1}}{(k-1)!}\left(\frac{|z|^2}{4r}\right)^{k-1}  \bigg]  f'(r) \, \d r
\,\d z \d u\\
&=\beta(d-2+\beta) \int_0^t \int_{\Rd} g(u,x,z) |z|^{\beta-2} \, \d z \d u.
\end{align*}
Now, we can also use the dominated convergence theorem to take $\beta \to 2n$ in \eqref{eq:2}.
The base case $n = 1$ follows by the same
argument. This completes the inductive argument.
\end{example}

\section*{Acknowledgements}
The research is partially supported by Grant No. 2023/49/B/ST1/00678 of the National Science Center (Poland).

\small
\bibliographystyle{abbrv}

\end{document}